\documentclass[aop,preprint]{imsart}

\RequirePackage[OT1]{fontenc}
\RequirePackage{amsthm,amsmath}
\RequirePackage[numbers]{natbib}
\RequirePackage[colorlinks,citecolor=blue,urlcolor=blue]{hyperref}

\usepackage{amsthm,amsmath,verbatim,amssymb}

\arxiv{arXiv:1806.01555}

\startlocaldefs
\theoremstyle{plain}
\newtheorem{thm}{Theorem}[section]
\newtheorem{lem}{Lemma}[section]

\newtheorem{cor}{Corollary}[section]

\endlocaldefs

\newcommand\numberthis{\addtocounter{equation}{1}\tag{\theequation}}

\usepackage{color}
\allowdisplaybreaks

\begin{document}

\begin{frontmatter}
\title{Small gaps of circular $\beta$-ensemble}
\runtitle{Small gaps}

\begin{aug}
\author{\fnms{Renjie} \snm{Feng}\thanksref{m1}\ead[label=e1]{renjiefeng.math@gmail.com}},
\and
\author{\fnms{Dongyi} \snm{Wei}\thanksref{m2}\ead[label=e3]{jnwdyi@pku.edu.cn}}

\runauthor{Feng and Wei}

\affiliation{University of Science and Technology of China\thanksmark{m1}}
 \affiliation{Peking University\thanksmark{m2}}

\address{University of Science and Technology\\ of China, Hefei, China, 230026.\\
\phantom{E-mail:\ }\printead*{e1}}

\address{Peking University, Beijing, China, 100871.\\
\phantom{E-mail:\ }\printead*{e3}}

\end{aug}

\begin{abstract}
In this article, we study the smallest gaps of the log-gas $\beta$-ensemble on the unit circle (C$\beta$E),  where $\beta$ is any positive integer.
The main result is that the smallest gaps, after being normalized by $n^{\frac {\beta+2}{\beta+1}}$, will converge in distribution  to a Poisson point process with some
explicit intensity. And thus one can derive the limiting density of the $k$-th smallest gap,  which is proportional to $x^{k(\beta+1)-1}e^{-x^{\beta+1}}$. In particular, the result applies
to the classical COE, CUE and CSE in random matrix theory.
The essential part of the proof is to derive several  identities and inequalities regarding the Selberg integral, which should have their own interest.
\end{abstract}

%

\end{frontmatter}

\section{Introduction}

The extreme spacings  of random point processes are important quantities in
statistical physics. In random matrix theory,  the question regarding the smallest gaps  of CUE and GUE was considered by Vinson \cite{V}; by a different method, Soshnikov also investigated the smallest gaps for the determinantal
point processes on the real line with  translation invariant kernels \cite{So};  Soshnikov's  technique was adapted by   Ben Arous-Bourgade in \cite{BB} where they proved
that the smallest gaps of CUE and GUE, after being normalized by $n^{4/3}$, will tend to a Poisson point process and  the $k$-th smallest gap has the limiting density proportional to $x^{3k-1}e^{-x^{3}}$. Their results are further generalized by Figalli-Guionnet in \cite{FG}. The similar results are derived for random matrices with complex Ginibre,
Wishart and universal Unitary ensembles in \cite{SJ}.

Regarding the largest gaps,  the decay order $\sqrt{32\log n}/n$ of the largest gaps of CUE and GUE (in the bulk regime) was predicted by Vinson  in \cite{V} and proved by Ben Arous-Bourgade in \cite{BB}. The same decay order for the largest gaps of some invariant multimatrix Hermitian matrices was also derived by Figalli-Guionnet in \cite{FG}. Recently, the fluctuations of the largest gaps of CUE and GUE have been derived in \cite{FW2}, furthermore, it's proved that the largest gaps, after being normalized, will tend to a Poisson point process.


 In this paper, we will derive the limitinig distribution of the smallest gaps of C$\beta$E where $\beta$ is any positive integer.  Our results confirm the (numerical) prediction in physics \cite{STKZ} and recover Ben Arous-Bourgade's results in the case of CUE (where $\beta=2$). But our proof is different and technical. One can not make use of the structure of the determinantal point processes any more (for example,
 when $\beta=1, 4$, they  are Pfaffian processes other than the determinantal point processes \cite{AGZ}),  and we have to start from the Selberg integral to get the estimates regarding the point correlation functions, where we  need to derive
 several asymptotic limits and inequalities (such as Lemma \ref{lem7} and Lemma \ref{lem9}) which should have their own interest in Selberg integral theory. The method developed in this paper is further adapted in  \cite{FTW} where we can derive the limiting distribution of the smallest gaps of GOE.

Recently, in \cite{B, ds},  Bourgade and  Landon-Lopatto-Marcinek further proved that our  results are universal for both small gaps and large gaps in the bulk of the general Hermitian and symmetric Wigner matrices with assumptions.


\subsection{Main results}
For circular $ \beta$-ensemble with  $\beta>0$, the density of the eigenangles $ \theta_j\in[-\pi,\pi)$,$1\leq j\leq n$ with respect to the Lebesgue measure is\begin{align}\label{ji}&J(\theta_1,\cdots, \theta_n)=\frac{1}{C_{\beta,n}}\prod_{j<k}|e^{i\theta_j}-e^{i\theta_k}|^{\beta}
\end{align}with $ \beta=2$ corresponding to CUE and $ \beta=1$ for COE and $ \beta=4$ for CSE. The partition function
\begin{align*}&{C_{\beta,n}}:=\int_{-\pi}^{\pi}d\theta_1\cdots\int_{-\pi}^{\pi}d\theta_n
\prod_{j<k}|e^{i\theta_j}-e^{i\theta_k}|^{\beta}
\end{align*} is derived by the Selberg integral as \begin{align*}&{C_{\beta,n}}=(2\pi)^n\frac{\Gamma(1+\beta n/2)}{(\Gamma(1+\beta /2))^n}.
\end{align*}

 One interpretation of the density $J(\theta_1,\cdots, \theta_n)$ is as the Boltzmann factor for a classical gas at inverse temperature
$ \beta$ with potential energy$$-\sum_{1\leq j<k\leq n}\ln|e^{i\theta_j}-e^{i\theta_k}|.$$ Because of the pairwise logarithmic repulsion, such a classical
gas is referred to as a log-gas. This interpretation allows for a number of properties of correlations
and distributions to be anticipated using arguments based on macroscopic electrostatics \cite{For}.

We will need the following partition functions for the two-component log-gas where the system consists of $n_1$ particles with charge $q=1$ and $n_2$ particles with charge $q=2$, \begin{equation}\label{c1}{C_{\beta,n_1,n_2}}:=\int_{-\pi}^{\pi}d\theta_1\cdots\int_{-\pi}^{\pi}d\theta_{n_1+n_2}
\prod_{j<k}|e^{i\theta_j}-e^{i\theta_k}|^{q_jq_k\beta}\end{equation} and
\begin{equation}\label{c2}{C_{\beta,n_1,n_2}}(I):=\int_{(-\pi,\pi)^{n_1}\times I^{n_2}}d\theta_{1}\cdots d\theta_{n_1+n_2}
\prod_{j<k}|e^{i\theta_j}-e^{i\theta_k}|^{q_jq_k\beta},
\end{equation}where $q_j=1$ for $1\leq j\leq n_1$ and $q_j=2$ for $n_1+1\leq j\leq n_1+n_2.$

We also need the following partition function with respect to the two-component log-gas with $n_1$ particles with charge $q=1$ and one particle with charge $q=k$, \begin{align}\label{c3}{C_{\beta,n_1,(k)}}:=\int_{-\pi}^{\pi}d\theta_1\cdots\int_{-\pi}^{\pi}d\theta_{n_1+1}
\prod_{j<l}|e^{i\theta_j}-e^{i\theta_l}|^{q_jq_l\beta}
\end{align}with $q_j=1$ for $1\leq j\leq n_1$ and $q_{n_1+1}=k$, then we have $$C_{\beta,n_1,(2)}=C_{\beta,n_1,1} $$ and the following results.\begin{lem}\label{lem7}For $0<k\leq n,\ \beta\geq 1$, we have\begin{align*}&{C_{\beta,n-k,(k)}}\leq {C_{\beta,n}(n\beta)^{k(k-1)\beta/2}},
\end{align*} and \begin{align*}&\lim_{n\to+\infty}\frac{C_{\beta,n-2,1}}{C_{\beta,n}n^{\beta}}=A_{\beta},\ \ \ \ \lim_{n\to+\infty}\frac{C_{\beta,n-k,(k)}}{C_{\beta,n}n^{k(k-1)\beta/2}}=A_{\beta,k},
\end{align*}where\begin{align*}&
A_{\beta,k}=\frac{(2\pi)^{1-k}(\Gamma(\beta/2+1))^{k}}
{\Gamma(k\beta/2 +1)}\prod_{j=1}^{k-1}
\frac{\Gamma(j\beta/2 +1)}{\Gamma((k+j)\beta/2 +1)}(\beta/2)^{k(k-1)\beta/2}
\end{align*}and\begin{align*}&A_{\beta}=A_{\beta,2}=(2\pi)^{-1}\frac{(\beta /2)^{\beta}(\Gamma(\beta/2+1))^{3}}{\Gamma(3\beta/2 +1)\Gamma(\beta+1)}.
\end{align*}\end{lem}

 Now we consider the following point process on $ \mathbb{R}^2$\begin{align}\label{chi}&\chi^{(n,\gamma)}=\sum_{i=1}^n\delta_{\left(n^{\gamma}(\theta_{(i+1)}-\theta_{(i)}),\theta_{(i)}\right)},\ \ \ \ \chi^{(n)}=\chi^{(n,\gamma)}\Big|_{\gamma=\frac{\beta+2}{\beta+1}},
\end{align}where $\gamma>0$, $\theta_{(i)}\ (1\leq i\leq n)$ is the increasing rearrangement of $\theta_{i}\ (1\leq i\leq n)$ and
$\theta_{(i+n)}=\theta_{(i)}+2\pi,$ i.e. the indexes are modulo $n.$  Regarding the point process $\chi^{(n)}$,  the main result is \begin{thm}\label{thm1}
For C$\beta$E where $ \beta$ is a positive integer,
the process $\chi^{(n)} $  will converge to a Poisson point process $ \chi$ as $n\to+\infty$ with intensity\begin{align*}&\mathbb{E}\chi(A\times I)= \frac{A_{\beta}|I|}{2\pi}\int_Au^{\beta}du,
\end{align*}  where $ A\subset\mathbb{R}_{+}$ is any bounded Borel set, $I\subseteq(-\pi,\pi)$ and $|I|$ is the Lebesgue measure of $I$. In particular,  the result holds for COE, CUE and CSE  with $$A_1=\frac 1{24},\,\,A_2=\frac 1{24\pi},\,\,A_4=\frac 1{270\pi}$$ respectively.
\end{thm}
As a direct consequence of the main result, we easily have (we  refer to \cite{BB, V} for the case when $\beta=2$)
\begin{cor}Let $t_k$ be the $k$-th smallest gap and we define $$ \tau_k=n^{(\beta+2)/(\beta+1)}\times\\(A_{\beta}/(\beta+1))^{{1}/({\beta+1})}t_k,$$ then we have
$$\lim_{n\to+\infty}\mathbb P(\tau_k\in A)=\int_A\frac{\beta+1}{(k-1)!} x^{k(\beta+1)-1}e^{-x^{\beta+1}}dx$$
 for any bounded interval $A\subset\mathbb{R}_{+}$.
\end{cor}


\subsection{Factorial moments and correlation functions}
We first review some basic concepts about the factorial moments and the correlation functions of a point process. 
Let $$X=\sum_{i}\delta_{X_i}$$ be a simple point process on $\mathbb R$, consider the point process
$$X^{(k)}=\sum_{X_{i_1},\cdots, X_{i_k}\mbox{all distinct}}\delta_{(X_{i_1},\cdots, X_{i_k})}$$
on $\mathbb R^k$. One can define   a measure $m_k$ on $\mathbb R^k$ by $$m_k(A)=\mathbb E(X^{(k)}(A))$$
for any Borel set $A$ in $\mathbb R^k$. If $m_k$ is absolutely continuous with respect to the Lebesgue measure, then there exists a function $f_k$ on $\mathbb R^k$ such that for any Borel sets $B_1, \cdots, B_k$ in $\mathbb R$, we have 
$$m_k(B_1\times \cdots \times B_k)=\int_{B_1\times \cdots \times B_k}f_k(x_1,\cdots, x_k)dx_1\cdots dx_k.$$
$f_k$ is called the $k$-point correlation function of the point process. Note that $f_k$ is not a probability density, but it admits the following probabilistic interpretation: for distinct points $x_1,\cdots, x_k$ in $\mathbb R$,  if $[x_ i, x_i+dx_i ], i=1, \cdots, k$ are neighbourhoods of $x_i$, then 
$f_k(x_1,\cdots, x_k)dx_1\cdots dx_k$ is the probability of the event that each set $[x_ i, x_i+dx_i ]$ contains a particle.

Moreover, one can check that  the $k$th factorial moment of a point process and the $k$-point correlation function satisfy
$$m_k(B^k)=\mathbb E \left(\frac{(X(B)!)}{(X(B)-k)!}\right)=\int_{B^k}f_k(x_1,\cdots, x_k)dx_1\cdots dx_k,$$
where $B$ is a Borel set in $\mathbb R$. 

If $X$ is a determinantal point process, then the $k$-point correlation function has the representation 
\begin{equation}\label{fk}f_k(x_1,\cdots, x_k)=\det [K(x_i, x_j)]_{1\leq i, j\leq k}\end{equation}
where $K(x,y)$ is a symmetric kernel. For example, in the case of CUE which is   a Haar measure on the unitary group U(n) with the joint density given in \eqref{ji} with $\beta=2$, the $k$-point correlation function is 
$$f_k(\theta_1,\cdots, \theta_k)=\det [K_n(\theta_i-\theta_j)]_{1\leq i, j\leq k},\,\,\,\, K_n(\theta)=\frac 1{2\pi}\frac{\sin (n\theta/2)}{\sin (\theta/2)}.$$
More properties regarding the correlation functions of determinantal point processes can be found in \cite{So2}. 

\subsection{Strategy and key lemmas}
Now we explain the main steps   to prove Theorem \ref{thm1}.
As in \cite{BB, So},  we still need to reduce the problem to the convergence of the factorial moments of $\chi^{(n)}$, but the proof  follows a quite different way.
This is because, for the determinantal point processes as considered in \cite{BB, So}, there are many structures one can make use of.  For example, all the point
correlation functions of the determinantal point precesses are given explicitly by symmetric kernels as in \eqref{fk} and one can express the factorial moments in terms of these
correlation functions,  and thus one can use Hadamard-Fischer inequality to control the estimates. But for general C$\beta$E, they are not  determinantal point
processes, one can only express the point correlation functions as  integrals of the joint density, and this causes many difficulties and all the proofs require delicate estimates of the integrals.

By the moment method, Theorem \ref{thm1} will be proved if we can prove the following convergence of the factorial moment \begin{align}\label{goal1}&\lim_{n\to+\infty}\mathbb{E}\left(\frac{( {\chi}^{(n)}(A\times I))!}{( {\chi}^{(n)}(A\times I)-k)!}\right)
=\left(\int_Au^{\beta}du\right)^k\left(\frac{|I|A_{\beta}}{2\pi}\right)^k
\end{align}
 for any fixed positive integer $k$, where   $ A\subset\mathbb{R}_{+}$ is any bounded interval and $I\subseteq(-\pi,\pi)$.

We will not prove this convergence directly. We will study the following auxiliary point process instead.
We now introduce $ \theta_{i,j}=\theta_i-\theta_j$ for $ \theta_i>\theta_j,$  $ \theta_{i,j}=\theta_i-\theta_j+2\pi$ for $ \theta_i<\theta_j.$ For any $\gamma>0$,  we  define \begin{equation}\label{thetar} \theta_{i,j,\gamma}=(n^{\gamma}\theta_{i,j},\theta_j)\end{equation}and \begin{align}\label{chit}&\widetilde{\chi}^{(n,\gamma)}=\sum_{i\neq j}\delta_{\theta_{i,j,\gamma}},\ \widetilde{\chi}^{(n)}=\widetilde{\chi}^{(n,\gamma)}\Big|_{\gamma=\frac{\beta+2}{\beta+1}},
\end{align} i.e., $\widetilde{\chi}^{(n)}$ is the point process of all normalized spacings, then   we have \begin{equation}\chi^{(n)}\leq\widetilde{\chi}^{(n)}. \end{equation} In fact,  we can rewrite \begin{equation}\widetilde{\chi}^{(n,\gamma)}=\sum\limits_{j=1}^{n-1}\widetilde{\chi}^{(n,\gamma,j)} \end{equation} such that\begin{align}&\widetilde{\chi}^{(n,\gamma,j)}=\sum_{i=1}^n\delta_{\left(n^{\gamma}(\theta_{(i+j)}-\theta_{(i)}),\theta_{(i)}\right)}.
\end{align}Then we have $$\widetilde{\chi}^{(n,\gamma,1)}= {\chi}^{(n,\gamma)}\,\,\mbox{and} \,\, 0\leq\widetilde{\chi}^{(n,\gamma,j)}(B)\leq n $$ for every Borel set $ B\subset\mathbb{R}^{2}$.

We need to show the following  lemma which indicates that there is no successive smallest gaps, which is also considered in \cite{BB, So} for the determinantal point processes.\begin{lem}\label{lem4} For any bounded interval  $ A\subset\mathbb{R}_{+}$ and $I\subseteq(-\pi,\pi),$  we have $\chi^{(n)}(A\times I)-\widetilde{\chi}^{(n)}(A\times I)\to0$ in probability  as $n\to+\infty$.\end{lem}

The proof of Lemma \ref{lem4} replies on the estimates of the integration of the  3-point correlation functions. For the determinantal point processes as considered in \cite{BB, So},  all the point correlation functions can be expressed in terms of the determinant of the   kernels, and thus all the estimates  follow from the estimates of the   kernels. But in our case, we can only express   the correlation functions as the integrations of the joint density and thus we will need several integral inequalities as in Lemma \ref{lem10} in \S \ref{ii}. These inequalities will be applied many times in the whole proof.  

The significance of the above lemma is that, instead of proving the convergence of the factorial moment of $\chi^{(n)}$ in \eqref{goal1}, it's enough to prove the following convergence of the factorial moment of $\widetilde\chi^{(n)}$ of all normalized spacings \begin{align}\label{20}&\lim_{n\to+\infty}\mathbb{E}\left(\frac{(\widetilde{\chi}^{(n)}(A\times I))!}{(\widetilde{\chi}^{(n)}(A\times I)-k)!}\right)
=\left(\int_Au^{\beta}du\right)^k\left(\frac{|I|A_{\beta}}{2\pi}\right)^k
\end{align} for any fixed $k$. 
 Actually,  \eqref{20}  is the direct consequence of  the following Lemma \ref{lem8} and Lemma \ref{lem9}.

\begin{lem}\label{lem8}For any bounded interval $A\subset\mathbb{R}^+$,  $I\subseteq(-\pi,\pi)$ and
any positive integer $k\geq 1$,   we have\begin{align*}&
\mathbb{E}\left(\frac{(\widetilde{\chi}^{(n)}(A\times I))!}{(\widetilde{\chi}^{(n)}(A\times I)-k)!}\right)
-\left(\int_Au^{\beta}du\right)^k\frac{C_{\beta,n-2k,k}(I)}{C_{\beta,n}n^{k\beta}}\to 0
\end{align*}as $n\to+\infty$.\end{lem}

To prove Lemma \ref{lem8}, we will introduce another auxiliary point process  $$\rho^{(k,n,\gamma)}=\sum_{i_1,\cdots,i_{2k}\ \text{all distinct}}\delta_{\left(\theta_{i_1,i_2,\gamma},\cdots,\theta_{i_{2k-1},i_{2k},\gamma}\right)},\ \
\rho^{(k,n)}=\rho^{(k,n,\gamma)}\Big|_{\gamma=\frac{\beta+2}{\beta+1}},
$$
where $\theta_{i_{2j-1},i_{2j}, \gamma}$ ($1\leq j\leq k$) is  defined in \eqref{thetar}.

Regarding $\rho^{(k,n)}$, we will see that the expectation of $\rho^{(k,n)}$  will converge to the $k$-th factorial moment of
$\widetilde\chi^{(n)}$. To be more precise,  Lemma \ref{lem8} is the consequence of the following two convergences
$$\lim_{n\to+\infty}\left(\mathbb{E}\frac{(\widetilde{\chi}^{(n)}(A\times I))!}{(\widetilde{\chi}^{(n)}(A\times I)-k)!}
-\mathbb{E}\rho^{(k,n)}((A\times I)^k)\right)=0$$and $$
\lim_{n\to+\infty}\left(\mathbb{E}(\rho^{(k,n)}((A\times I)^k))
-\left(\int_Au^{\beta}du\right)^k\frac{C_{\beta,n-2k,k}(I)}{C_{\beta,n}n^{k\beta}}\right)=0.
$$
Here,  the second limit is the most significant part and indicates the idea of the whole proof,  it implies the bounds of $ \mathbb{E}(\rho^{(k,n)}((A\times I)^k)$ by the quotient of the partition
functions $C_{\beta,n-2k,k}(I)/(C_{\beta,n}n^{k\beta}),$ therefore,  the problem regarding the smallest gaps in nature is
just a problem about integral estimates.
To be more precise, one of the  crucial ideas of the whole method is that one can bound  $\mathbb{E} \left(\rho^{(k,n)}((A\times I)^k)\right)$
which is expressed in terms of the integral of the joint density of  the one-component log-gas (see \eqref{aaaaa}) by the generalized partition function of
the two-component log-gas (see \eqref{bbbbb} and Lemma \ref{lem12}). 

The intuitive idea of the whole proof is natural: for a pair of two particles  with charge 1  of the smallest gap  of C$\beta$E, these two particles will tend to a ``double particle'' with charge 2 in the limit, therefore,  
 if there are $k$-pair of such particles among $n$ particles in one-component log-gas, then such system can be approximated by   two-component log-gas with $n-2k$ particles with charge $1$ and $k$ particles with charge $2$, therefore,  one needs to compare the partition function of one-component log-gas with the partition function of the two-component log-gas as in the following lemma. 
\begin{lem}\label{lem9}For any interval $I\subseteq(-\pi,\pi)$ and
any positive integer $k\geq 1$, we have\begin{align*}&\lim_{n\to+\infty}\frac{C_{\beta,n-2k,k}(I)}{C_{\beta,n}n^{k\beta}}
=\left(\frac{|I|A_{\beta}}{2\pi}\right)^k.
\end{align*}\end{lem}
The convergence for $k=1$ is guaranteed by Lemma \ref{lem7}.
In \S\ref{ccccc}, we will prove  Lemma \ref{lem9}  by induction based on Lemma \ref{lem8} and the following inequality
\begin{equation}\label{19}\limsup_{n\to+\infty}\frac{C_{\beta,n-4,2}(I)}{C_{\beta,n}n^{2\beta}}
\leq\left(\frac{|I|A_{\beta}}{2\pi}\right)^2.
\end{equation}


 The proof of the upper bound \eqref{19} is complicated and it will be proved in \S\ref{19a} based on the properties of Selberg
integral and generalized hypergeometric functions derived in \cite{For}, here a key point is the limit in Lemma \ref{lem6}.


 In a recent paper \cite{FTW}, the method developed in this article is further applied to derive the limiting distribution of the smallest gaps of GOE. Actually our method is quite general, it can be used to prove that of G$\beta$E  and more general ensembles.
In all cases, as indicated by the intuitive idea mentioned above, one of the main difficulties to study the smallest gaps  is to prove the analogue asymptotic limit as in Lemma \ref{lem9}, i.e., one has to prove the asymptotic limit of the quotient of the two-component log-gas and one-component log-gas, once this  is done, the smallest gaps can be proved to be converging to a Poisson distribution and hence the limiting density can be derived. 

As a  final remark, we also conjecture that Theorem \ref{thm1} must be true for any $\beta>0$, but our method only works for the positive integer $\beta$. This is because,
in the proof of the upper bound \eqref{19}, we use properties of generalized hypergeometric functions that are valid for positive integer $\beta$. As explained above,
if one can prove Lemma \ref{lem9} for every $\beta>0$ by other method without using the properties of generalized hypergeometric functions, then Theorem \ref{thm1}
will hold for every  $\beta>0$.

\bigskip
\textbf{Acknowledgement:} We are indebted to the anonymous reviewers for providing many corrections and insightful comments, this paper would not have been possible without their supportive work.

\section{Proof of  Lemma \ref{lem7}}
Now we give the proof of Lemma \ref{lem7}, which is  based on the Selberg
integral. We refer to Selberg's original method \cite{Sel} and Aomoto's method \cite{A, A1} for the proof of the Selberg integral. We also refer to Chapter 4 in \cite{For} for other proofs and several applications of the Selberg integral, especially in random matrix theory. 

\begin{proof} We can write \begin{align*}&{C_{\beta,n_1,1}}\\=&\int_{-\pi}^{\pi}d\theta_1\cdots\int_{-\pi}^{\pi}d\theta_{n_1+1}
\prod_{1\leq j<k\leq n_1}|e^{i\theta_j}-e^{i\theta_k}|^{\beta} \prod_{1\leq j\leq n_1}|e^{i\theta_j}-e^{i\theta_{n_1+1}}|^{2\beta}\\=&\int_{-\pi}^{\pi}d\theta_1\cdots\int_{-\pi}^{\pi}d\theta_{n_1+1}
\prod_{1\leq j<k\leq n_1}|e^{i\theta_j}-e^{i\theta_k}|^{\beta}\prod_{1\leq j\leq n_1}|e^{i\theta_j}+1|^{2\beta}\\=&(2\pi)^{n_1+1}M_{n_1}(\beta,\beta,\beta/2),
\end{align*}here we used changing of variables $\theta_j\mapsto\theta_j+\theta_{n_1+1}\pm \pi\ (1\leq j\leq n_1) $ and the formula (4.4) in \cite{For}:\begin{align}\label{mn}M_{n}(a,b,\lambda)&:=\int_{-1/2}^{1/2}d\theta_1\cdots\int_{-1/2}^{1/2}d\theta_n\prod_{l=1}^ne^{\pi i\theta_l(a-b)}|1+e^{2\pi i\theta_l}|^{a+b}\times \nonumber\\&\prod_{1\leq j<k\leq n}|e^{2\pi i\theta_j}-e^{2\pi i\theta_k}|^{2\lambda} \nonumber\\&=\prod_{j=0}^{n-1}\frac{\Gamma(\lambda j+a+b+1)\Gamma(\lambda (j+1)+1)}{\Gamma(\lambda j+a+1)\Gamma(\lambda j+b+1)\Gamma(1+\lambda)}.
\end{align}Similarly, we have \begin{align}\label{32}&{C_{\beta,n_1,1}}(I)=(2\pi)^{n_1}|I|M_{n_1}(\beta,\beta,\beta/2)=(2\pi)^{-1}|I|{C_{\beta,n_1,1}},
\end{align}and\begin{align}\label{33}&{C_{\beta,n_1,(k)}}=(2\pi)^{n_1+1}M_{n_1}(k\beta/2,k\beta/2,\beta/2).
\end{align} For every positive integer $k$, we have\begin{align*}M_{n}(k\lambda,k\lambda,\lambda)=&\prod_{j=0}^{n-1}
\frac{\Gamma(\lambda (j+2k)+1)\Gamma(\lambda (j+1)+1)}{(\Gamma(\lambda (j+k)+1))^2\Gamma(1+\lambda)}\\=&\frac{1}{\Gamma(\lambda+1)^n}\prod_{j=k}^{2k-1}
\frac{\Gamma(\lambda (n+j)+1)}{\Gamma(j\lambda +1)}\prod_{j=1}^{k-1}\frac{\Gamma(j\lambda +1)}{\Gamma(\lambda (n+j)+1)},
\end{align*}thus we have \begin{align*}{C_{\beta,n_1,(k)}}=&(2\pi)^{n_1+1}M_{n_1}(k\beta/2,k\beta/2,\beta/2)\\=&\frac{(2\pi)^{n_1+1}}
{(\Gamma(\beta/2+1))^{n_1}}
\prod_{j=k}^{2k-1}
\frac{\Gamma(\beta (n_1+j)/2+1)}{\Gamma(j\beta/2 +1)}\prod_{j=1}^{k-1}\frac{\Gamma(j\beta/2 +1)}{\Gamma(\beta (n_1+j)/2+1)}.
\end{align*}And for $n_1=n-k>0$, we have\begin{align*}\frac{C_{\beta,n-k,(k)}}{C_{\beta,n}}=&\frac{(2\pi)^{1-k}(\Gamma(\beta/2+1))^{k}}
{\Gamma(n\beta/2 +1)}
\prod_{j=1}^{2k-1}\left(\frac{\Gamma(\beta (n_1+j)/2+1)}{\Gamma(j\beta/2 +1)}\right)^{\text{sgn}(j-k)}\\&=\frac{(2\pi)^{1-k}(\Gamma(\beta/2+1))^{k}}
{\Gamma(k\beta/2 +1)}
\prod_{j=1}^{k-1}
\frac{\Gamma(\beta (n+j)/2+1)\Gamma(j\beta/2 +1)}{\Gamma((k+j)\beta/2 +1)\Gamma(\beta (n-j)/2+1)}.
\end{align*}As $\ln \Gamma(x)$ is convex for $x>0,$ we have $(\Gamma(\beta/2+1))^{k}\leq \Gamma(k\beta/2 +1)$. For $n>k-1\geq j\geq 1$, we have $k\beta/2\geq 1,\ \beta j\geq 1 $ and \begin{align*}&\frac{\Gamma(j\beta/2 +1)}{\Gamma((k+j)\beta/2 +1)}\leq\left(\frac{\Gamma(j\beta/2 +1)}{\Gamma(j\beta/2 +2)}\right)^{k\beta/2}=\left(\frac{1}{j\beta/2 +1}\right)^{k\beta/2}\leq 1
\end{align*}and\begin{align*}&\frac{\Gamma(\beta (n+j)/2+1)}{\Gamma(\beta (n-j)/2+1)}\leq\left(\frac{\Gamma(\beta (n+j)/2+1)}{\Gamma(\beta (n+j)/2)}\right)^{\beta j}=\left(\beta (n+j)/2\right)^{\beta j}\leq \left(n\beta \right)^{\beta j},
\end{align*}therefore, we have \begin{align*}&\frac{C_{\beta,n-k,(k)}}{C_{\beta,n}}\leq(2\pi)^{1-k}
\prod_{j=1}^{k-1}\left(n\beta \right)^{\beta j}=(2\pi)^{1-k}
\left(n\beta \right)^{k(k-1)\beta/2},
\end{align*}which will imply the first inequality. Using convexity of $\ln \Gamma(x)$, we also have\begin{align*}&\left(\beta (n-j)/2+1\right)^{\beta j}\leq\frac{\Gamma(\beta (n+j)/2+1)}{\Gamma(\beta (n-j)/2+1)}\leq\left(\beta (n+j)/2\right)^{\beta j},
\end{align*}which implies\begin{align*}&\lim_{n\to+\infty}\frac{\Gamma(\beta (n+j)/2+1)}{\Gamma(\beta (n-j)/2+1)n^{\beta j}}=\left(\beta/2\right)^{\beta j}.
\end{align*}And thus we have \begin{align*} &\lim_{n\to+\infty}\frac{C_{\beta,n-k,(k)}}{C_{\beta,n}n^{k(k-1)\beta/2}}\\=&\frac{(2\pi)^{1-k}(\Gamma(\beta/2+1))^{k}}
{\Gamma(k\beta/2 +1)}
\prod_{j=1}^{k-1}
\frac{\Gamma(j\beta/2 +1)}{\Gamma((k+j)\beta/2 +1)}\cdot \lim_{n\to+\infty}\prod_{j=1}^{k-1}\frac{\Gamma(\beta (n+j)/2+1)}{\Gamma(\beta (n-j)/2+1)n^{\beta j}}\\=&\frac{(2\pi)^{1-k}(\Gamma(\beta/2+1))^{k}}
{\Gamma(k\beta/2 +1)}
\prod_{j=1}^{k-1}
\frac{\Gamma(j\beta/2 +1)}{\Gamma((k+j)\beta/2 +1)}\prod_{j=1}^{k-1}\left(\beta/2\right)^{\beta j}=:A_{\beta,k}.
\end{align*}  As $C_{\beta,n_1,(2)}=C_{\beta,n_1,1}, $ we have \begin{align*}&\lim_{n\to+\infty}\frac{C_{\beta,n-2,1}}{C_{\beta,n}n^{\beta}}=\lim_{n\to+\infty}
\frac{C_{\beta,n-2,(2)}}{C_{\beta,n}n^{\beta}}=A_{\beta,2},
\end{align*}and the expression of $ A_{\beta}=A_{\beta,2}$ follows directly from that of $A_{\beta,k}. $\end{proof}

\section{One more auxiliary point process}\label{aux}
Now we can introduce another auxiliary point process as  \begin{align}\label{rhoo}&\rho^{(k,n,\gamma)}=\sum_{i_1,\cdots,i_{2k}\ \text{all distinct}}\delta_{\left(\theta_{i_1,i_2,\gamma},\cdots,\theta_{i_{2k-1},i_{2k},\gamma}\right)}
\end{align}
and we define 
\begin{equation}\label{rhood}\rho^{(k,n)}=\rho^{(k,n,\gamma)}\Big|_{\gamma=\frac{\beta+2}{\beta+1}},\end{equation}
where $\theta_{i_{2j-1},i_{2j}, \gamma}$ ($1\leq j\leq k$) is defined in \eqref{thetar}.

We first have the following  lemma which will be used to prove that the expectation of the random variable $\rho^{(k,n)} $ converges to the factorial moment of
$\widetilde{\chi}^{(n)} $ (see \eqref{25} below). 

\begin{lem}\label{lem14} For any bounded intervals  $ A\subset\mathbb{R}_{+}$ and $I\subseteq(-\pi,\pi),$ let $B=A\times I$, then we have \begin{align*}&\rho^{(k,n,\gamma)}(B^k)\leq\frac{(\widetilde{\chi}^{(n,\gamma)}(B))!}{(\widetilde{\chi}^{(n,\gamma)}(B)-k)!},\,\,\,\gamma>0.
\end{align*}Let $c_1$ be such that $A\subset(0,c_1)$, $c_n=c_1n^{-\frac{\beta+2}{\beta+1}}$ and \begin{align*}&a=\max\left\{i-j:i,j\in\mathbb{Z},\ \theta_{(i)}-\theta_{(j)}\leq 2c_n \right\},
\end{align*}if $c_n\in (0,1)$, then we have\begin{align*}&0\leq\frac{(\widetilde{\chi}^{(n)}(B))!}{(\widetilde{\chi}^{(n)}(B)-k)!}-\rho^{(k,n)}(B^k)
\leq k(k-1)(a-1)(\widetilde{\chi}^{(n)}(B))^{k-1}
\end{align*}and\begin{align*}&\rho^{(k,n)}(B^k)\geq(\widetilde{\chi}^{(n)}(B))^k-k(k-1)a(\widetilde{\chi}^{(n)}(B))^{k-1}.
\end{align*}\end{lem}\begin{proof}We denote \begin{align*}&X_1=\{(i_1,\cdots,i_{2k}):i_j\in\mathbb{Z}, 1\leq i_j\leq n,\ \forall\ 1\leq j\leq 2k,\\ &i_{2j-1}\neq i_{2j},\ \forall\ 1\leq j\leq k, \{i_{2j-1}, i_{2j}\}\neq \{i_{2l-1}, i_{2l}\}, \ \forall\ 1\leq j<l\leq k\},\\&X_2=\{(i_1,\cdots,i_{2k}):i_j\in\mathbb{Z}, 1\leq i_j\leq n,\ \forall\ 1\leq j\leq 2k,\\ &i_{j}\neq i_{l},\  \ \forall\ 1\leq j<l\leq 2k\},\\&Y_{j,l}=\{(i_1,\cdots,i_{2k}):\{i_{2j-1}, i_{2j}\}\cap \{i_{2l-1}, i_{2l}\}\neq \emptyset\},
\end{align*}then we have $X_2\subseteq X_1$ and $X_1\setminus X_2=\cup_{1\leq j<l\leq k}Y_{j,l}.$ Let\begin{align*}&X_{j,B}=\{(i_1,\cdots,i_{2k})\in X_j:\theta_{i_{2j-1}, i_{2j},\gamma}\in B,\ \forall\ 1\leq j\leq k\},\ j=1,2,\\&Y_{j,l,B}=\{(i_1,\cdots,i_{2k})\in Y_{j,l}:\theta_{i_{2j-1}, i_{2j},\gamma}\in B,\ \forall\ 1\leq j\leq k\},
\end{align*}then we have \begin{equation}\label{rdg}\rho^{(k,n,\gamma)}(B^k)=|X_{2,B}|,\  X_{2,B}\subseteq X_{1,B}, |X_{1,B}|= \dfrac{(\widetilde{\chi}^{(n,\gamma)}(B))!}{(\widetilde{\chi}^{(n,\gamma)}(B)-k)!},\end{equation} which gives the first inequality, here $|X|$ is the cardinality of the set $X.$

 We also have $X_{1,B}\setminus X_{2,B}=\cup_{1\leq j<l\leq k}Y_{j,l,B}$ and by symmetry $|Y_{j,l,B}|=|Y_{1,2,B}| $ for $ 1\leq j<l\leq k,$ therefore \begin{align}\label{38}&|X_{1,B}|-| X_{2,B}|\leq\sum_{1\leq j<l\leq k}|Y_{j,l,B}|=k(k-1)|Y_{1,2,B}|/2.
\end{align}

Now we assume $ \gamma=\dfrac{\beta+2}{\beta+1}.$ If $a=0,$ then we have $\theta_{j,l}\geq n^{-\gamma}(2c_n)=2c_1 $ for every $1\leq j<l\leq n,$ thus  $\theta_{j,l,\gamma}\not\in B, $ and $\widetilde{\chi}^{(n)}(B)=\rho^{(k,n)}(B^k)=0;$ if $k=1$, then  by definition $\widetilde{\chi}^{(n)}(B)=\rho^{(k,n)}(B^k)$. Thus the second and third inequalities are clearly true in these two trivial cases, for the rest,  we only need to consider the case $a>0,k>1.$ The key point is to estimate $|Y_{1,2,B}|. $

For fixed $\theta_{i_{1}, i_{2},\gamma}\in B, $ we will show that there are at most $2(a-1) $ choices of $(i_3,i_4) $ to satisfy $ (i_1,\cdots,i_{2k})\in Y_{1,2,B} .$ Let \begin{align*}&T_j=\{l:l\neq i_j,\theta_{i_{j}, l,\gamma}\in B\}\cup\{l:l\neq i_j,\theta_{l,i_{j},\gamma}\in B\},\\ &T_j'=\{l:l\neq i_j,\theta_{i_{j}, l}\in (0,c_n)\}\cup\{l:l\neq i_j,\theta_{l,i_{j}}\in (0,c_n)\},\ j=1,2.
\end{align*}Then we have $T_j\subseteq T_j',$ since $\theta_{j,l,\gamma}\in B $ implies $n^{\gamma}\theta_{j,l}\in A\subset(0,c_1) $ and $\theta_{j,l}\in (0,n^{-\gamma}c_1)=(0,c_n) .$ Assume $ \theta_{i_1}=\theta_{(p)}$ then we have\begin{align*}\{\theta_l:l\in T_1'\cup \{i_1\}\}&=\{\theta_{(q)}(\text{mod}2\pi):|\theta_{(q)}-\theta_{(p)}|<c_n\}\\&=\{\theta_{(q)}(\text{mod}2\pi):r\leq q\leq s\},
\end{align*}for some $r,s\in \mathbb{Z}$ such that $|\theta_{(r)}-\theta_{(p)}|<c_n,\ |\theta_{(s)}-\theta_{(p)}|<c_n, $ therefore $|\theta_{(r)}-\theta_{(s)}|<2c_n ,$ and by definition of $a$ we have $s-r\leq a.$ Since $i_1\not\in T_1'$, we have\begin{align*}|T_1'|+1&=|\{\theta_l:l\in T_1'\cup \{i_1\}\}|=|\{\theta_{(q)}(\text{mod}2\pi):r\leq q\leq s\}|\\&\leq s-r+1\leq a+1,
\end{align*}and thus $|T_1|\leq |T_1'|\leq a$. Similarly, we have $|T_2|\leq |T_2'|\leq a.$

 Now for $\theta_{i_{1}, i_{2},\gamma}\in B, $ by definition we have $ i_2\in T_1$ and $i_1\in T_2.$

 If $\theta_{i_{3}, i_{4},\gamma}\in B, \ \{i_{1}, i_{2}\}\cap \{i_{3}, i_{4}\}\neq \emptyset,\ \{i_{1}, i_{2}\}\neq \{i_{3}, i_{4}\}$, then we must have $\{i_{3}, i_{4}\}=\{i_{1}, l\},\ l\in T_2\setminus\{i_1\} $ or $\{i_{3}, i_{4}\}=\{i_{2}, l\},\ l\in T_1\setminus\{i_2\}, $ and the order of $i_3,i_4$ is uniquely determined. In fact, by the definition of $\theta_{i,j}$,  we have $\theta_{i_{3}, i_{4}}+\theta_{i_{4}, i_{3}}=2\pi$, if $\theta_{i_{3}, i_{4},\gamma}\in B,\ \theta_{i_{4}, i_{3},\gamma}\in B$ then we have $n^{\gamma}\theta_{i_{3}, i_{4}},n^{\gamma}\theta_{i_{4}, i_{3}}\in A\subset(0,c_1),$ and $\theta_{i_{3}, i_{4}}+\theta_{i_{4}, i_{3}}<2n^{-\gamma}c_1=2c_n<2\pi, $ a contradiction.

 Thus for $\theta_{i_{1}, i_{2},\gamma}\in B, $ the number of $(i_3,i_4)$ satisfying $\theta_{i_{3}, i_{4},\gamma}\in B, \ \{i_{1}, i_{2}\}\cap \{i_{3}, i_{4}\}\neq \emptyset,\ \{i_{1}, i_{2}\}\neq \{i_{3}, i_{4}\}$ is at most $|T_2\setminus\{i_1\}|+|T_1\setminus\{i_2\}|=|T_2|-1+|T_1|-1\leq 2(a-1). $
 Now there are $\widetilde{\chi}^{(n)}(B) $ choices of $(i_{1},i_{2}),$ for fixed $(i_{1},i_{2}) $ there are at most $2(a-1) $ choices of $(i_3,i_4) $ and $\widetilde{\chi}^{(n)}(B) $ choices of $(i_{2l-1},i_{2l}),\ 3\leq l\leq k, $ to satisfy $ (i_1,\cdots,i_{2k})\in Y_{1,2,B} ,$ thus we have \begin{align*}|Y_{1,2,B}|\leq \widetilde{\chi}^{(n)}(B)\times2(a-1)\times\widetilde{\chi}^{(n)}(B)^{k-2}=2(a-1)\widetilde{\chi}^{(n)}(B)^{k-1}.
\end{align*}By \eqref{rdg} and \eqref{38}, we have \begin{align*}0&\leq\frac{(\widetilde{\chi}^{(n)}(B))!}{(\widetilde{\chi}^{(n)}(B)-k)!}-\rho^{(k,n)}(B^k)=|X_{1,B}|-| X_{2,B}|\leq k(k-1)|Y_{1,2,B}|/2
\\&\leq k(k-1)(a-1)(\widetilde{\chi}^{(n)}(B))^{k-1},
\end{align*}which is the second inequality.

The third inequality follows from the second inequality and the fact that\begin{align*}&\frac{(\widetilde{\chi}^{(n)}(B))!}{(\widetilde{\chi}^{(n)}(B)-k)!}
=\prod_{j=0}^{k-1}(\widetilde{\chi}^{(n)}(B)-j)=(\widetilde{\chi}^{(n)}(B))^k\prod_{j=0}^{k-1}(1-j/\widetilde{\chi}^{(n)}(B))\\ \geq& (\widetilde{\chi}^{(n)}(B))^k\left(1-\sum_{j=0}^{k-1}j/\widetilde{\chi}^{(n)}(B)\right)\\=&(\widetilde{\chi}^{(n)}(B))^k-k(k-1) (\widetilde{\chi}^{(n)}(B))^{k-1}/2,
\end{align*}this completes the proof.\end{proof}

\section{Integral inequalities}\label{ii}
 In this section, we will  prove one integral lemma regarding the upper and lower bounds of the integration of the joint density on the neighborhood around one variable. As a direct consequence, we can derive
 several integral inequalities about the two-component log-gas.
 \subsection{Integral lemma}
 We first prove the following lemma which will be applied many times in the whole proof.
  \begin{lem}\label{lem10} Let $m,n,\beta$
be positive integers with $m\leq n$. Given any $c$ such that $n\beta c\in(0,1)$ and $\theta_j\in\mathbb{R}$, $j=1,\cdots, m$, we define $$F(x)=\prod_{j=1}^m(e^{ix}-e^{i\theta_j}),$$ then we have\begin{align*}& \left(\frac{\sin(c/2)}{c/2}\right)^{\beta}\cos(n\beta c)\frac{c^{\beta+1}}{\beta+1}\int_{-\pi}^{\pi}d x_1|F(x_1)|^{2\beta}\\ \leq&
\int_{-\pi}^{\pi}d x_1\int_{x_1}^{x_1+c}d x_2|e^{ix_1}-e^{ix_2}|^{\beta}|F(x_1)|^{\beta}|F(x_2)|^{\beta}\\ \leq& \frac{c^{\beta+1}}{\beta+1}\int_{-\pi}^{\pi}d x_1|F(x_1)|^{2\beta},
\end{align*}and for $k\geq 1,$ we have \begin{align*}&
\int_{-\pi}^{\pi}d x_1\int_{(x_1,x_1+c)^{k-1}}d x_2\cdots dx_k\prod_{1\leq j<l\leq k}|e^{ix_j}-e^{ix_l}|^{\beta}\prod_{j=1}^k|F(x_j)|^{\beta}\\ &\leq c^{\beta k(k-1)/2+k-1}\int_{-\pi}^{\pi}d x_1|F(x_1)|^{k\beta}.
\end{align*}For intervals $A\subset(0,c),\ I\subset(-\pi,\pi),$ we denote $$\varphi({\beta},A):=\int_A|1-e^{iu}|^{\beta}du, $$ then we have\begin{align*}&
\left|\int_{I}d x_1\int_{x_1+A}d x_2|e^{ix_1}-e^{ix_2}|^{\beta}|F(x_1)|^{\beta}|F(x_2)|^{\beta}-\varphi({\beta},A)\int_{I}d x_1|F(x_1)|^{2\beta}\right|\\& \leq \varphi({\beta},A)(n\beta c)\int_{-\pi}^{\pi}d x_1|F(x_1)|^{2\beta}
\end{align*}and\begin{align*}&
\ \ \left(\frac{\sin(c/2)}{c/2}\right)^{\beta }\int_Au^{\beta}du\leq\varphi({\beta},A)\leq\int_Au^{\beta}du.
\end{align*}\end{lem}\begin{proof}We can write $$F(x)^{\beta}=\sum\limits_{j=0}^{m\beta}a_je^{ijx}.$$ A change of variables $x_2=x_1+t$ shows \begin{align}\label{37}&\int_{-\pi}^{\pi}d x_1\int_{x_1}^{x_1+c}d x_2|e^{ix_1}-e^{ix_2}|^{\beta}|F(x_1)|^{\beta}|F(x_2)|^{\beta}\\ \nonumber
=&\int_0^cdt\int_{-\pi}^{\pi}|1-e^{it}|^{\beta}|F(x_1)|^{\beta}|F(x_1+t)|^{\beta}d x_1.
\end{align}As\begin{align*}&F(x_1)^{\beta}=\sum\limits_{j=0}^{m\beta}a_je^{ijx_1},\ \ F(x_1+t)^{\beta}=\sum\limits_{j=0}^{m\beta}a_je^{ijt}e^{ijx_1},
\end{align*}by Parseval's theorem, we have \begin{align*}&\int_{-\pi}^{\pi}\overline{F(x_1)^{\beta}}F(x_1+t)^{\beta}d x_1=2\pi\sum_{j=0}^{m\beta}\overline{a_j}a_je^{ijt}=2\pi\sum_{j=0}^{m\beta}|{a_j}|^2e^{ijt}
\end{align*}and\begin{align*}&\int_{-\pi}^{\pi}|F(x_1)|^{2\beta}d x_1=\int_{-\pi}^{\pi}|F(x_1)^{\beta}|^2d x_1=2\pi\sum_{j=0}^{m\beta}|{a_j}|^2.
\end{align*}Thus for $t\in(0,c),\ 0\leq j\leq m\beta\leq n\beta$, we have $0\leq jt\leq n\beta c<1$ and\begin{align}\label{35}&\int_{-\pi}^{\pi}|F(x_1)|^{\beta}|F(x_1+t)|^{\beta}d x_1\\\nonumber\geq&\text{Re}\int_{-\pi}^{\pi}\overline{F(x_1)^{\beta}}F(x_1+t)^{\beta}d x_1\\ \nonumber=&2\pi\sum_{j=0}^{m\beta}|{a_j}|^2(\cos{jt})\geq 2\pi\sum_{j=0}^{m\beta}|{a_j}|^2\cos({n\beta c})\\ \nonumber=&\cos({n\beta c})\int_{-\pi}^{\pi}|F(x_1)|^{2\beta}d x_1,
\end{align}integrating for $t\in(0,c)$ gives\begin{align*}&\int_{-\pi}^{\pi}d x_1\int_{x_1}^{x_1+c}d x_2|e^{ix_1}-e^{ix_2}|^{\beta}|F(x_1)|^{\beta}|F(x_2)|^{\beta}\\
\geq&\int_0^cdt|1-e^{it}|^{\beta}\cos({n\beta c})\int_{-\pi}^{\pi}|F(x_1)|^{2\beta}d x_1.
\end{align*}As $(\sin x)/x$ is decreasing for $x\in (0,1)$ and $0<c\leq n\beta c<1$, we further have\begin{align*}&\int_0^cdt|1-e^{it}|^{\beta}=\int_0^cdt|2\sin (t/2)|^{\beta}\geq \int_0^cdt\left|t\frac{\sin(c/2)}{c/2}\right|^{\beta}=\frac{c^{\beta+1}}{\beta+1}\left(\frac{\sin(c/2)}{c/2}\right)^{\beta}.
\end{align*}Therefore, we have \begin{align*}&\int_{-\pi}^{\pi}d x_1\int_{x_1}^{x_1+c}d x_2|e^{ix_1}-e^{ix_2}|^{\beta}|F(x_1)|^{\beta}|F(x_2)|^{\beta}\\
\geq&\frac{c^{\beta+1}}{\beta+1}\left(\frac{\sin(c/2)}{c/2}\right)^{\beta}\cos({n\beta c})\int_{-\pi}^{\pi}|F(x_1)|^{2\beta}d x_1,
\end{align*}which is the lower bound in the first inequality.

On the other hand, since $F$ is $2\pi$-perodic, for $t\in(0,c)$, we have \begin{align*}&0\leq\int_{-\pi}^{\pi}\left||F(x_1)|^{\beta}-|F(x_1+t)|^{\beta}\right|^2d x_1\\&=\int_{-\pi}^{\pi}(|F(x_1)|^{2\beta}+|F(x_1+t)|^{2\beta})d x_1-2\int_{-\pi}^{\pi}|F(x_1)|^{\beta}|F(x_1+t)|^{\beta}d x_1\\&=2\int_{-\pi}^{\pi}|F(x_1)|^{2\beta}d x_1-2\int_{-\pi}^{\pi}|F(x_1)|^{\beta}|F(x_1+t)|^{\beta}d x_1,
\end{align*}which implies\begin{align}\label{36}&\int_{-\pi}^{\pi}|F(x_1)|^{\beta}|F(x_1+t)|^{\beta}dx_1 \leq\int_{-\pi}^{\pi}|F(x_1)|^{2\beta}d x_1,
\end{align}and using \eqref{35} and $2-2\cos({n\beta c})\leq ({n\beta c})^2 $, we also have \begin{align}\label{34}\int_{-\pi}^{\pi}\left||F(x_1)|^{\beta}-|F(x_1+t)|^{\beta}\right|^2d x_1\leq({n\beta c})^2\int_{-\pi}^{\pi}|F(x_1)|^{2\beta}d x_1.
\end{align}By \eqref{37} and \eqref{36}, we have\begin{align*}&\int_{-\pi}^{\pi}d x_1\int_{x_1}^{x_1+c}d x_2|e^{ix_1}-e^{ix_2}|^{\beta}|F(x_1)|^{\beta}|F(x_2)|^{\beta}\\
\leq&\int_0^cdt|1-e^{it}|^{\beta}\int_{-\pi}^{\pi}|F(x_1)|^{2\beta}d x_1\\
\leq&\int_0^cdt|t|^{\beta}\int_{-\pi}^{\pi}|F(x_1)|^{2\beta}d x_1\\=&\frac{c^{\beta+1}}{\beta+1}\int_{-\pi}^{\pi}d x_1|F(x_1)|^{2\beta},
\end{align*}which gives the upper bound in the first inequality.

 If $x_j\in(x_1,x_1+c)$ for $1<j\leq k,$ then we have $|e^{ix_j}-e^{ix_l}|\leq|x_j-x_l|<c$ for $ 1\leq j<l\leq k,$ therefore,\begin{align*}&
\int_{-\pi}^{\pi}d x_1\int_{(x_1,x_1+c)^{k-1}}d x_2\cdots dx_k\prod_{1\leq j<l\leq k}|e^{ix_j}-e^{ix_l}|^{\beta}\prod_{j=1}^k|F(x_j)|^{\beta}\\ \leq&
\int_{-\pi}^{\pi}d x_1\int_{(x_1,x_1+c)^{k-1}}d x_2\cdots dx_k\prod_{1\leq j<l\leq k}c^{\beta}\prod_{j=1}^k|F(x_j)|^{\beta}\\ =&
c^{\beta k(k-1)/2}\int_{(0,c)^{k-1}}d t_2\cdots dt_k\int_{-\pi}^{\pi}d x_1\prod_{j=1}^k|F(x_1+t_j)|^{\beta}\\ \leq&
\frac{c^{\beta k(k-1)/2}}{k}\int_{(0,c)^{k-1}}d t_2\cdots dt_k\int_{-\pi}^{\pi}d x_1\sum_{j=1}^k|F(x_1+t_j)|^{k\beta}\\ =&
\frac{c^{\beta k(k-1)/2}}{k}\sum_{j=1}^k\int_{(0,c)^{k-1}}d t_2\cdots dt_k\int_{-\pi}^{\pi}d x_1|F(x_1)|^{k\beta}\\ =& c^{\beta k(k-1)/2+k-1}\int_{-\pi}^{\pi}d x_1|F(x_1)|^{k\beta},
\end{align*}which is the second inequality, here we denote $t_1=0.$

By changing of variables, the definition of $\varphi({\beta},A)$, H\"older inequality and \eqref{34}, we have\begin{align*}&
\left|\int_{I}d x_1\int_{x_1+A}d x_2|e^{ix_1}-e^{ix_2}|^{\beta}|F(x_1)|^{\beta}|F(x_2)|^{\beta}-\varphi({\beta},A)\int_{I}d x_1|F(x_1)|^{2\beta}\right|\\=&\left|\int_A du\int_{I}d x_1|1-e^{iu}|^{\beta}|F(x_1)|^{\beta}|F(x_1+u)|^{\beta} -\int_A |1-e^{iu}|^{\beta}du\int_{I}d x_1|F(x_1)|^{2\beta}\right|\\ \leq & \int_A du\int_{I}d x_1|1-e^{iu}|^{\beta}|F(x_1)|^{\beta}\left||F(x_1+u)|^{\beta}-|F(x_1)|^{\beta}\right|\\ \leq & \int_A du|1-e^{iu}|^{\beta}\left(\int_{I}d x_1|F(x_1)|^{2\beta}\right)^{\frac{1}{2}}\times \left(\int_{I}d x_1\left||F(x_1+u)|^{\beta}-|F(x_1)|^{\beta}\right|^2\right)^{\frac{1}{2}}\\ \leq  & \int_A du|1-e^{iu}|^{\beta}\left(\int_{-\pi}^{\pi}d x_1|F(x_1)|^{2\beta}\right)^{\frac{1}{2}}\left(({n\beta c})^2\int_{-\pi}^{\pi}|F(x_1)|^{2\beta}d x_1\right)^{\frac{1}{2}}\\=& \varphi({\beta},A)(n\beta c)\int_{-\pi}^{\pi}d x_1|F(x_1)|^{2\beta},
\end{align*}which is the third inequality.

As $(\sin x)/x$ is decreasing for $x\in (0,1)$ and $$A\subset(0,c)\subset(0,1),$$ we have\begin{align*}\varphi({\beta},A)&=\int_A|1-e^{iu}|^{\beta}du=\int_A|2\sin (u/2)|^{\beta}du \\&\geq \int_A\left|u\frac{\sin(c/2)}{c/2}\right|^{\beta}du=\left(\frac{\sin(c/2)}{c/2}\right)^{\beta}\int_Au^{\beta}du,
\end{align*}and as $ |1-e^{iu}|\leq u,$ we also have\begin{align*}&\varphi({\beta},A)=\int_A|1-e^{iu}|^{\beta}du\leq\int_Au^{\beta}du,
\end{align*}which gives the fourth inequality. This completes the proof.\end{proof}

\subsection{Inequalities regarding two-component log-gas}
Let  $B=(0,c_0)\times (-\pi,\pi),\ n>2k,$ by definition of $\rho^{(k,n,\gamma)}$ (recall \eqref{rhoo}), we have \begin{equation}\label{times}\mathbb{E}\rho^{(k,n,\gamma)}(B^k)=\frac{n!}{(n-2k)!}\int_{\Sigma_{n,k,c}}J(\theta_1,\cdots, \theta_n)d\theta_1\cdots d\theta_n\Big|_{c=c_0/n^{\gamma}},
\end{equation}here \begin{align*}\numberthis \label{sigma}\Sigma_{n,k,c}=\big\{&(\theta_1,\cdots, \theta_n):\theta_j\in(-\pi,\pi),\forall\ 1\leq j\leq n-k,\\& \theta_j-\theta_{j-k}\in(0,c),\forall\ n-k< j\leq n\big\}.
\end{align*}

For $0\leq l\leq k$, with assumptions in Lemma \ref{lem10},  we denote \begin{equation}\label{enkl}{E_{\beta,n,k,l}}(c):=\int_{\Sigma_{n-l,k-l,c}}d\theta_1\cdots d\theta_{n-l}
\prod_{j<m}|e^{i\theta_j}-e^{i\theta_m}|^{q_jq_m\beta}\Big|_{q_s=1+\chi_{\{s\leq l\}}}.
\end{equation}

Then we have\begin{align*}&\int_{\Sigma_{n,k,c}}J(\theta_1,\cdots, \theta_n)d\theta_1\cdots d\theta_n=\frac{{E_{\beta,n,k,0}}(c)}{C_{\beta,n}},
\end{align*}and by definition we can check that $${E_{\beta,n,k,k}}(c)=C_{\beta,n-2k,k}.
$$ We need to show that (for $0< n\beta c<1$)\begin{align}\label{21}&\left(\frac{\sin(c/2)}{c/2}\right)^{\beta}\cos(n\beta c)\frac{c^{\beta+1}}{\beta+1}\leq\frac{{E_{\beta,n,k,l-1}}(c)}{{E_{\beta,n,k,l}}(c)}\leq\frac{c^{\beta+1}}{\beta+1}.
\end{align}In fact, after changing the order of variables, we can write\begin{align*}&{E_{\beta,n,k,l-1}}(c)=\int_{\Sigma_{n-l-1,k-l,c}}d\theta_1\cdots d\theta_{n-l-1}
\prod_{1\leq j<m\leq n-l-1}|e^{i\theta_j}-e^{i\theta_m}|^{q_jq_m\beta}\\ &\times\int_{-\pi}^{\pi}d x_1\int_{x_1}^{x_1+c}d x_2|e^{ix_1}-e^{ix_2}|^{\beta}\prod_{j=1}^2\prod_{m=1}^{n-l-1}|e^{ix_j}-e^{i\theta_m}|^{q_m\beta}\Big|_{q_s=1+\chi_{\{s\leq l-1\}}},
\end{align*}and \begin{align*}&{E_{\beta,n,k,l}}(c)=\int_{\Sigma_{n-l-1,k-l,c}}d\theta_1\cdots d\theta_{n-l-1}
\prod_{1\leq j<m\leq n-l-1}|e^{i\theta_j}-e^{i\theta_m}|^{q_jq_m\beta}\\ &\times\int_{-\pi}^{\pi}d x_1\prod_{m=1}^{n-l-1}|e^{ix_1}-e^{i\theta_m}|^{2q_m\beta}\Big|_{q_s=1+\chi_{\{s\leq l-1\}}},
\end{align*}then \eqref{21} is the direct consequence of Lemma \ref{lem10} by taking  $$F(x)=\prod_{m=1}^{n-l-1}(e^{ix}-e^{i\theta_m})^{q_m}.$$

 By \eqref{21} we finally have the following two estimates \begin{align}\label{22}&{{E_{\beta,n,k,l}}(c)}\leq\left(\frac{c^{\beta+1}}{\beta+1}\right)^{k-l}{E_{\beta,n,k,k}}(c)
=\left(\frac{c^{\beta+1}}{\beta+1}\right)^{k-l}C_{\beta,n-2k,k}
\end{align}and\begin{align}\label{23}&\left(\frac{\sin(c/2)}{c/2}\right)^{k\beta}\left(\cos(n\beta c)\right)^{k}\left(\frac{c^{\beta+1}}{\beta+1}\right)^k{C_{\beta,n-2k,k}}\leq{{E_{\beta,n,k,0}}(c)}.
\end{align}

\section{No successive small gaps}
In this section, we will prove Lemma \ref{lem4} which implies that there is no successive smallest gaps.   We first need the following estimate.
 \begin{lem}\label{lem11} For $B=(0,c_0)\times (-\pi,\pi),\ n\geq k>1,\ n^{1-\gamma}\beta c_0\in(0,1)$, we have \begin{align*}&\mathbb{E}\widetilde{\chi}^{(n,\gamma,k-1)}(B)\leq n(n^{1-\gamma}\beta c_0)^{\beta k(k-1)/2+k-1}.
\end{align*}\end{lem}\begin{proof}We consider the point process\begin{align*}&\xi^{(n)}=\sum_{i=1}^n\delta_{\theta_{i}},\ \ \ \ \xi^{(n,k)}=\sum_{i_1,\cdots,i_{k}\ \text{all distinct}}\delta_{(\theta_{i_1},\cdots,\theta_{i_{k}})}.
\end{align*}For $B=(0,c_0)\times (-\pi,\pi),\ n\geq k>1,$ let $c_n=c_0/n^{\gamma},$ then we have\begin{align*}&\widetilde{\chi}^{(n,\gamma,j)}(B)=\sum_{i=1}^n \bold 1_{\xi^{(n)}(\theta_i+(0,c_n))\geq j}\leq\frac{1}{j!}{\xi^{(n,j+1)}(\Lambda_{j+1,c_n})},
\end{align*}here, the angles are modulo $2\pi$, $\bold 1$ is the indicator of an event and we define \begin{align*}\Lambda_{k,c}=\big\{&(\theta_1,\cdots, \theta_k):\theta_1\in(-\pi,\pi),\ \theta_j-\theta_{1}\in(0,c),\forall\ 1< j\leq k\big\}.
\end{align*} Let\begin{align*}\Lambda_{k,c,n}=\big\{&(\theta_1,\cdots, \theta_n):\theta_j\in(-\pi,\pi),\forall\ 1\leq j\leq n-k+1,\\ &\theta_j-\theta_{n-k+1}\in(0,c),\forall\ n-k+1< j\leq n\big\},
\end{align*}then by Lemma \ref{lem7} and Lemma \ref{lem10}, we have\begin{align*}&\mathbb{E}\widetilde{\chi}^{(n,\gamma,k-1)}(B)\leq\frac{1}{(k-1)!}\mathbb{E}{\xi^{(n,k)}(\Lambda_{k,c_n})}\\
=&\frac{1}{(k-1)!}\frac{n!}{(n-k)!}\int_{\Lambda_{k,c_n,n}}J(\theta_1,\cdots, \theta_n)d\theta_1\cdots d\theta_n\\
=&\frac{1}{(k-1)!}\frac{n!}{(n-k)!}\frac{1}{C_{\beta,n}}\int_{-\pi}^{\pi}d\theta_1\cdots \int_{-\pi}^{\pi}d\theta_{n-k}\prod_{1\leq j<m\leq n-k}|e^{i\theta_j}-e^{i\theta_m}|^{\beta}\\ &\times\int_{\Lambda_{k,c_n}}d x_1\cdots dx_k\prod_{1\leq j<m\leq k}|e^{ix_j}-e^{ix_m}|^{\beta}\prod_{j=1}^k\prod_{m=1}^{n-k}|e^{ix_j}-e^{i\theta_m}|^{\beta}\\
\leq&\frac{n^k}{(k-1)!}\frac{1}{C_{\beta,n}}\int_{-\pi}^{\pi}d\theta_1\cdots \int_{-\pi}^{\pi}d\theta_{n-k}\prod_{1\leq j<m\leq n-k}|e^{i\theta_j}-e^{i\theta_m}|^{\beta}\\ &\times c_n^{\beta k(k-1)/2+k-1}\int_{-\pi}^{\pi}d x_1\prod_{m=1}^{n-k}|e^{ix_1}-e^{i\theta_m}|^{k\beta}\\=&\frac{n^k}{(k-1)!}\frac{C_{\beta,n-k,(k)}}{C_{\beta,n}}c_n^{\beta k(k-1)/2+k-1}\\ \leq& \frac{n^k}{(k-1)!}(n\beta)^{k(k-1)\beta/2}c_n^{\beta k(k-1)/2+k-1}=\frac{n(n\beta c_n)^{\beta k(k-1)/2+k-1}}{(k-1)!\beta^{k-1}}\\ \leq &n(n\beta c_n)^{\beta k(k-1)/2+k-1}=n(n^{1-\gamma}\beta c_0)^{\beta k(k-1)/2+k-1},
\end{align*}this completes the proof.\end{proof}

Now we can  give the proof of Lemma \ref{lem4}.\begin{proof}Let $c$ be such that $A\subset(0,c)$, and $B=(0,c)\times (-\pi,\pi),\ \gamma=\dfrac{\beta+2}{\beta+1}.$ Then by definitions \eqref{chi} and \eqref{chit},  $\chi^{(n)}(A\times I)-\widetilde{\chi}^{(n)}(A\times I)\neq 0$ implies $\widetilde{\chi}^{(n,\gamma,j)}(A\times I)>0 $ for some $j>1$,  and thus we must have $\widetilde{\chi}^{(n,\gamma,2)}(B)>0. $ Since $ \gamma>1,$ for $n$ large enough we have $ n^{1-\gamma}\beta c\in(0,1),$ and by Lemma \ref{lem11} with $k=3$, we have \begin{align*}&\mathbb{P}(\chi^{(n)}(A\times I)-\widetilde{\chi}^{(n)}(A\times I)\neq 0)\leq \mathbb{P}(\widetilde{\chi}^{(n,\gamma,2)}(B)>0)\\ \leq& \mathbb{E}(\widetilde{\chi}^{(n,\gamma,2)}(B))\leq n(n^{1-\gamma}\beta c)^{3\beta +2}=n(n^{-\frac{1}{\beta+1}}\beta c)^{3\beta +2}\to 0,
\end{align*}this completes the proof.\end{proof}



\section{Proof of  Lemma \ref{lem8}}
In this section, we will prove Lemma \ref{lem8}.
\subsection{Uniform boundedness}
We will first prove the following uniform boundedness which will be applied in the proofs of Lemma \ref{lem8} and Lemma \ref{lem9}.
\begin{lem}
\begin{equation}\label{18} \limsup_{n\to+\infty}\frac{C_{\beta,n-2k,k}}{C_{\beta,n}n^{k\beta}}<+\infty.
\end{equation}\end{lem}
\begin{proof}Let $c_0$ be fixed such that $\beta c_0\in(0,1) $ and $B=(0,c_0)\times (-\pi,\pi).$
Thanks to the integral expression of $\mathbb{E}\rho^{(k,n,\gamma)}(B^k) $ in \eqref{times}, the definition of $E_{\beta,n,k,l}$  \eqref{enkl} and the upper bound \eqref{23}, with $\gamma=1$, we have \begin{align*}&\mathbb{E}\rho^{(k,n,1)}(B^k)=\frac{n!}{(n-2k)!}\frac{E_{\beta,n,k,0}(c)}{C_{\beta,n}}\Big|_{c=c_0/n}
 \\ \geq &\frac{n!}{(n-2k)!}\frac{C_{\beta,n-2k,k}}{C_{\beta,n}}\times
\left(\frac{\sin(c/2)}{c/2}\right)^{k\beta}\left(\cos(n\beta c)\right)^{k}\left(\frac{c^{\beta+1}}{\beta+1}\right)^k\Big|_{c=c_0/n}
\\=&\frac{n!n^{-k}}{(n-2k)!}\frac{C_{\beta,n-2k,k}}{C_{\beta,n}n^{k\beta}}\left(\frac{\sin(c_0/(2n))}{c_0/(2n)}\right)^{k\beta}
\left(\cos(\beta c_0)\right)^{k}\left(\frac{c_0^{\beta+1}}{\beta+1}\right)^k.
\end{align*}By the first inequality in Lemma \ref{lem14}, we have \begin{align*}&\rho^{(k,n,1)}(B^k)\leq\frac{(\widetilde{\chi}^{(n,1)}(B))!}{(\widetilde{\chi}^{(n,1)}(B)-k)!}\leq (\widetilde{\chi}^{(n,1)}(B))^k,
\end{align*}which implies\begin{align*}&\limsup_{n\to+\infty}\mathbb{E}(n^{-1}\widetilde{\chi}^{(n,1)}(B))^k
\\ \geq &\limsup_{n\to+\infty}n^{-k}\mathbb{E}\rho^{(k,n,1)}(B^k)
\\ \geq&\lim_{n\to+\infty}\frac{n!n^{-2k}}{(n-2k)!}\limsup_{n\to+\infty}\frac{C_{\beta,n-2k,k}}{C_{\beta,n}n^{k\beta}}
\left(\cos(\beta c_0)\right)^{k}\left(\frac{c_0^{\beta+1}}{\beta+1}\right)^k\\ =&\limsup_{n\to+\infty}\frac{C_{\beta,n-2k,k}}{C_{\beta,n}n^{k\beta}}
\left(\frac{c_0^{\beta+1}\cos(\beta c_0)}{\beta+1}\right)^k.
\end{align*}Thus, to prove  \eqref{18},  we only need to prove\begin{align}\label{24}&\limsup_{n\to+\infty}\mathbb{E}(n^{-1}\widetilde{\chi}^{(n,1)}(B))^k
<+\infty.
\end{align}As $\widetilde{\chi}^{(n,\gamma)}=\sum\limits_{j=1}^{n-1}\widetilde{\chi}^{(n,\gamma,j)}, $ by Lemma \ref{lem11} (since $\beta c_0\in (0,1)$), we have\begin{align*}\mathbb{E}(n^{-1}\widetilde{\chi}^{(n,1,j)}(B))\leq(\beta c_0)^{\beta j(j+1)/2+j}\leq(\beta c_0)^{j}.
\end{align*}Using $0\leq\widetilde{\chi}^{(n,1,j)}(B)\leq n, $ we have\begin{align*}\mathbb{E}(n^{-1}\widetilde{\chi}^{(n,1,j)}(B))^k\leq\mathbb{E}(n^{-1}\widetilde{\chi}^{(n,1,j)}(B))
\leq(\beta c_0)^{j}.
\end{align*}By Minkowski inequality, we finally have\begin{align*}(\mathbb{E}(n^{-1}\widetilde{\chi}^{(n,1)}(B))^k)^{1/k}&\leq \sum_{j=1}^{n-1}(\mathbb{E}(n^{-1}\widetilde{\chi}^{(n,1,j)}(B))^k)^{1/k}\leq \sum_{j=1}^{n-1}(\beta c_0)^{j/k}\\ &\leq(1-(\beta c_0)^{1/k})^{-1},
\end{align*}thus \eqref{24} is true, so is \eqref{18}.\end{proof}

\subsection{Proof of Lemma \ref{lem8}}
For $B=A\times I$, we will use  Lemma \ref{lem14} to deduce that\begin{align}\label{25}&\lim_{n\to+\infty}\left(\mathbb{E}\frac{(\widetilde{\chi}^{(n)}(B))!}{(\widetilde{\chi}^{(n)}(B)-k)!}
-\mathbb{E}\rho^{(k,n)}(B^k)\right)=0,\ \
\end{align}and use Lemma \ref{lem10} to deduce that\begin{align}\label{26}&
\lim_{n\to+\infty}\left(\mathbb{E}(\rho^{(k,n)}((A\times I)^k))
-\left(\int_Au^{\beta}du\right)^k\frac{C_{\beta,n-2k,k}(I)}{C_{\beta,n}n^{k\beta}}\right)=0,
\end{align}then Lemma \ref{lem8} follows from \eqref{25} and \eqref{26}, here $\rho^{(k,n)}$ is defined in \eqref{rhood}.

Let $ A\subset\mathbb{R}_{+}$ be any bounded interval,  $I\subseteq(-\pi,\pi)$ and $B=A\times I$. Let $c_1$ be such that $A\subset(0,c_1)$, and $B_1=(0,c_1)\times (-\pi,\pi)$ such that  $B\subset B_1.$ We denote $\gamma=\dfrac{\beta+2}{\beta+1}$ and $c_n=c_1/n^{\gamma}$.

Since $ \gamma>1,$ for $n$ large enough we have $ n\beta c_n=n^{1-\gamma}\beta c_1\in(0,1).$ By the expression of $\mathbb{E}\rho^{(k,n,\gamma)}(B^k) $, $E_{\beta,n,k,l} $ and \eqref{22},  with $\gamma(\beta+1)=\beta+2 $, we have \begin{align*}&\mathbb{E}\rho^{(k,n)}(B^k)\leq\mathbb{E}\rho^{(k,n)}(B_1^k)
=\frac{n!}{(n-2k)!}\frac{E_{\beta,n,k,0}(c_n)}{C_{\beta,n}}
\\ \leq&\frac{n!}{(n-2k)!}\frac{C_{\beta,n-2k,k}}{C_{\beta,n}}\left(\frac{c_n^{\beta+1}}{\beta+1}\right)^k\leq n^{2k}
\frac{C_{\beta,n-2k,k}}{C_{\beta,n}}\left(\frac{c_1^{\beta+1}}{\beta+1}\right)^kn^{-\gamma(\beta+1)k}
\\=&n^{2k}\frac{C_{\beta,n-2k,k}}{C_{\beta,n}}\left(\frac{c_1^{\beta+1}}{\beta+1}\right)^kn^{-(\beta+2)k}
=\frac{C_{\beta,n-2k,k}}{C_{\beta,n}n^{k\beta}}\left(\frac{c_1^{\beta+1}}{\beta+1}\right)^{k}.
\end{align*}Using \eqref{18}, we have\begin{align}\label{27}&\limsup_{n\to+\infty}\mathbb{E}\rho^{(k,n)}(B^k)<+\infty.
\end{align}Let $a$ be defined in Lemma \ref{lem14} and assume $n$ large enough such that $ 0<c_n\leq n\beta c_n=n^{1-\gamma}\beta c_1<1/4.$  By definition, we have $0\leq a<n$ and $a\geq k$ is equivalent to $\widetilde{\chi}^{(n,\gamma,k)}(B_2)>0$, here, $a,k\in\mathbb{Z},\ k>0$ and $B_2=(0,2c_1)\times (-\pi,\pi).$

By Lemma \ref{lem11} and $ (1-\gamma)(\beta +1)=-1 $, for $1\leq k<n$, we have\begin{align*}&\mathbb{P}(a\geq k)=\mathbb{P}(\widetilde{\chi}^{(n,\gamma,k)}(B_2)>0) \leq \mathbb{E}(\widetilde{\chi}^{(n,\gamma,k)}(B_2))\\ \leq& n(2n^{1-\gamma}\beta c_1)^{k(k+1)\beta/2 +k}=n(2n^{1-\gamma}\beta c_1)^{\beta +1}(2n^{1-\gamma}\beta c_1)^{(k+2)(k-1)\beta/2 +k-1}\\=&(2\beta c_1)^{\beta +1}(2n^{1-\gamma}\beta c_1)^{(k+2)(k-1)\beta/2 +k-1}\leq(2\beta c_1)^{\beta +1}(1/2)^{k-1}.
\end{align*} Since $\mathbb{P}(a\geq k)=0 $ for $k\geq n,$ thus \begin{align*}&\mathbb{P}(a\geq k)\leq(2\beta c_1)^{\beta +1}(1/2)^{k-1}
\end{align*}is always true for $k\geq 1$.

The above argument also implies that for $k>1,\ k\in \mathbb{Z}$, we must have
\begin{align*}&\lim_{n\to+\infty}\mathbb{P}(a\geq k)=0.
\end{align*}And by dominated convergence theorem, we can further deduce that \begin{align}\label{28}&\lim_{n\to+\infty}\mathbb{E}(a-1)_+^p=0,\ \forall\ p\in(0,+\infty),
\end{align} here, $f_+=\max(f,0).$

By Lemma \ref{lem14}, for any $k\geq 1$, we have $(\widetilde{\chi}^{(n)}(B))^k\leq2\rho^{(k,n)}(B^k)$ or $(\widetilde{\chi}^{(n)}(B))^k\leq2k(k-1)a(\widetilde{\chi}^{(n)}(B))^{k-1},
$ therefore, we have \begin{align*}&(\widetilde{\chi}^{(n)}(B))^k\leq\max(2\rho^{(k,n)}(B^k),(2k(k-1)a)^k)
\end{align*}and\begin{align*}&\mathbb{E}(\widetilde{\chi}^{(n)}(B))^k\leq2\mathbb{E}(\rho^{(k,n)}(B^k))+(2k(k-1))^k\mathbb{E}(a^k).
\end{align*}By \eqref{27} and \eqref{28}, we have  \begin{align}\label{29}&\limsup_{n\to+\infty}\mathbb{E}(\widetilde{\chi}^{(n)}(B))^k<+\infty.
\end{align}By Lemma \ref{lem14}, H\"older inequality, \eqref{28} and \eqref{29}, we have \begin{align*}0&\leq\mathbb{E}\left(\frac{(\widetilde{\chi}^{(n)}(B))!}{(\widetilde{\chi}^{(n)}(B)-k)!}-\rho^{(k,n)}(B^k)\right)
\\&\leq k(k-1)\mathbb{E}((a-1)_+(\widetilde{\chi}^{(n)}(B))^{k-1})\\&\leq k(k-1)(\mathbb{E}((a-1)_+^k))^{1/k}(\mathbb{E}(\widetilde{\chi}^{(n)}(B))^{k}))^{1-1/k}\to0
\end{align*}as $n\to+\infty,$ which implies \eqref{25}.

For $B=A\times I,\ n>2k,\gamma>0$, we have \begin{equation}\label{aaaaa}\mathbb{E}\rho^{(k,n,\gamma)}(B^k)=\frac{n!}{(n-2k)!}\int_{\Sigma_{n,k,cA,I}}J(\theta_1,\cdots, \theta_n)d\theta_1\cdots d\theta_n\Big|_{c=n^{-\gamma}},
\end{equation}here, \begin{align*}\Sigma_{n,k,A,I}=\big\{&(\theta_1,\cdots, \theta_n):\theta_j\in(-\pi,\pi),\forall\ 1\leq j\leq n-2k,\\ &\theta_{j-k}\in I, \theta_j-\theta_{j-k}\in A,\forall\ n-k< j\leq n\big\}.
\end{align*}We denote\begin{align*}&\Sigma_{n,k,A,I,l}=\big\{(\theta_1,\cdots, \theta_{n-l}):\theta_j\in(-\pi,\pi),\forall\ 1\leq j\leq n-2k,\\ &\theta_j\in I,\forall\ n-2k< j\leq n-k,\ \theta_j-\theta_{j-k+l}\in A,\forall\ n-k< j\leq n-l\big\}
\end{align*}and\begin{align*}&{E_{\beta,n,k,l}}(A,I):=\int_{\Sigma_{n,k,A,I,l}}d\theta_1\cdots d\theta_{n-l}
\prod_{j<p}|e^{i\theta_j}-e^{i\theta_p}|^{q_jq_p\beta}\end{align*} with ${q_s=1+\chi_{\{n-2k<s\leq n-2k+l\}}},$ then we have\begin{equation}\label{bbbbb}\int_{\Sigma_{n,k,A,I}}J(\theta_1,\cdots, \theta_n)d\theta_1\cdots d\theta_n=\frac{{E_{\beta,n,k,0}}(A,I)}{C_{\beta,n}}
\end{equation}and $${E_{\beta,n,k,k}}(A,I)=C_{\beta,n-2k,k}(I).
$$ We need inequalities similar to \eqref{21}.\begin{lem}\label{lem12} $ A\subset(0,c)$ and $I\subseteq(-\pi,\pi),$ $n\beta c\in(0,1),$ $n>2k,$ $n,\beta,k$ are positive integers, then we have \begin{align*}&\left|{E_{\beta,n,k,0}}(A,I)-\left(\int_Au^{\beta}du\right)^kC_{\beta,n-2k,k}(I)\right|\\ \leq&(kn\beta c+\beta kc^2/24)\left(\frac{c^{\beta+1}}{\beta+1}\right)^{k}C_{\beta,n-2k,k}.
\end{align*}\end{lem}\begin{proof}As before, after changing the order of variables, we can write
\begin{align*}&{E_{\beta,n,k,l-1}}(A,I)=\int_{\Sigma_{n-2,k-1,A,I,l-1}}d\theta_1\cdots d\theta_{n-l-1}
\Delta^{\beta}\\ &\times\int_{I}d x_1\int_{x_1+A}d x_2|e^{ix_1}-e^{ix_2}|^{\beta}\prod_{j=1}^2\prod_{m=1}^{n-l-1}|e^{ix_j}-e^{i\theta_m}|^{q_m\beta}
\end{align*}and \begin{align*}&{E_{\beta,n,k,l}}(A,I)=\int_{\Sigma_{n-2,k-1,A,I,l-1}}d\theta_1\cdots d\theta_{n-l-1}
\Delta^{\beta}\\ &\times\int_{I}d x_1\prod_{m=1}^{n-l-1}|e^{ix_1}-e^{i\theta_m}|^{2q_m\beta},
\end{align*}here,\begin{align*}&\Delta=
\prod_{1\leq j<m\leq n-l-1}|e^{i\theta_j}-e^{i\theta_m}|^{q_jq_m},\ {q_s=1+\chi_{\{n-2k<s< n-2k+l\}}}.
\end{align*}By Lemma \ref{lem10}, $ \Sigma_{n-2,k-1,A,I,l-1}\subset\Sigma_{n-l-1,k-l,c} $ and \eqref{22}, we have\begin{align*}&|{E_{\beta,n,k,l-1}}(A,I)-\varphi({\beta},A){E_{\beta,n,k,l}}(A,I)|\\ \leq &\varphi({\beta},A)(n\beta c)\times\int_{\Sigma_{n-2,k-1,A,I,l-1}}d\theta_1\cdots d\theta_{n-l-1}
\Delta^{\beta}\int_{-\pi}^{\pi}d x_1\prod_{m=1}^{n-l-1}|e^{ix_1}-e^{i\theta_m}|^{2q_m\beta}\\ \leq &\varphi({\beta},A)(n\beta c)\int_{\Sigma_{n-l-1,k-l,c}}d\theta_1\cdots d\theta_{n-l-1}\int_{-\pi}^{\pi}d x_1
\\&\prod_{1\leq j<m\leq n-l-1}|e^{i\theta_j}-e^{i\theta_m}|^{q_jq_m\beta}\prod_{m=1}^{n-l-1}|e^{ix_1}-e^{i\theta_m}|^{2q_m\beta}\Big|_{q_s=1+\chi_{\{0<s-n+2k< l\}}}\\&=\varphi({\beta},A)(n\beta c)\int_{\Sigma_{n-l,k-l,c}}d\theta_1\cdots d\theta_{n-l}\\& \times
 \prod_{1\leq j<m\leq n-l-1}|e^{i\theta_j}-e^{i\theta_m}|^{q_jq_m\beta}\Big|_{q_s=1+\chi_{\{0<s-n+2k\leq l\}}}\\&=\varphi({\beta},A)(n\beta c){E_{\beta,n,k,l}}(c)\\ &\leq \varphi({\beta},A)(n\beta c)\left(\frac{c^{\beta+1}}{\beta+1}\right)^{k-l}C_{\beta,n-2k,k},
\end{align*}where $\varphi(\beta, A)$ is as in Lemma \ref{lem10} and $E_{\beta,n,k,l}(c)$ is as in \eqref{enkl}.

Therefore (using Lemma \ref{lem10} again), we have \begin{align*}&|{E_{\beta,n,k,0}}(A,I)-\varphi({\beta},A)^k{E_{\beta,n,k,k}}(A,I)|\\ \leq&
\sum_{l=1}^k\varphi({\beta},A)^{l-1}|{E_{\beta,n,k,l-1}}(A,I)-\varphi({\beta},A){E_{\beta,n,k,l}}(A,I)|\\
\leq&\sum_{l=1}^k\varphi({\beta},A)^{l}(n\beta c)\left(\frac{c^{\beta+1}}{\beta+1}\right)^{k-l}C_{\beta,n-2k,k}\\
\leq&\sum_{l=1}^k(n\beta c)\left(\frac{c^{\beta+1}}{\beta+1}\right)^{k}C_{\beta,n-2k,k}=(kn\beta c)\left(\frac{c^{\beta+1}}{\beta+1}\right)^{k}C_{\beta,n-2k,k}.
\end{align*}As $1\geq\dfrac{\sin x}{x}\geq 1-x^2/6>0$ for $x\in (0,1)$, and by Lemma \ref{lem10}, we have\begin{align*}0\leq&\left(\int_Au^{\beta}du\right)^k-\varphi({\beta},A)^k\leq
\left(\int_Au^{\beta}du\right)^k\left(1-\left(\frac{\sin(c/2)}{c/2}\right)^{\beta k}\right)\\ \leq&
\left(\frac{c^{\beta+1}}{\beta+1}\right)^{k}\left(1-\left(1-c^2/24\right)^{\beta k}\right)\leq
\left(\frac{c^{\beta+1}}{\beta+1}\right)^{k}{\beta k}c^2/24.
\end{align*}By definition, we have\begin{align*}&0\leq{E_{\beta,n,k,k}}(A,I)=C_{\beta,n-2k,k}(I)\leq C_{\beta,n-2k,k},
\end{align*}therefore, we have\begin{align*}&\left|{E_{\beta,n,k,0}}(A,I)-\left(\int_Au^{\beta}du\right)^kC_{\beta,n-2k,k}(I)\right|\\ \leq&
\left|{E_{\beta,n,k,0}}(A,I)-\varphi({\beta},A)^k{E_{\beta,n,k,k}}(A,I)\right| +|\left(\int_Au^{\beta}du\right)^k-
\varphi({\beta},A)^k|C_{\beta,n-2k,k}(I)\\ \leq&(kn\beta c)\left(\frac{c^{\beta+1}}{\beta+1}\right)^{k}C_{\beta,n-2k,k}+\left(\frac{c^{\beta+1}}{\beta+1}\right)^{k}({\beta k}c^2/24)C_{\beta,n-2k,k},
\end{align*}which completes the proof.\end{proof}Now we ready to prove \eqref{26}. By the integral expression of $\mathbb{E}\rho^{(k,n,\gamma)}(B^k) $ with $\gamma=\frac{\beta+2}{\beta+1}$, the definition of $E_{\beta,n,k,l}(A, I) $ and changing of variables, we have\begin{align*}&
\mathbb{E}(\rho^{(k,n)}((A\times I))^k)
-\left(\int_Au^{\beta}du\right)^k\frac{C_{\beta,n-2k,k}(I)}{C_{\beta,n}n^{k\beta}}\\=&
\frac{n!}{(n-2k)!}\frac{{E_{\beta,n,k,0}}(n^{-\gamma}A,I)}{C_{\beta,n}}-\left(\int_{n^{-\gamma}A}u^{\beta}du\right)^k
\frac{C_{\beta,n-2k,k}(I)}{C_{\beta,n}n^{k\beta-k(\beta+1)\gamma}}\\=&
\frac{n^{2k}}{C_{\beta,n}}\left({E_{\beta,n,k,0}}({n^{-\gamma}A},I)-\left(\int_{n^{-\gamma}A}u^{\beta}du\right)^k
C_{\beta,n-2k,k}(I)\right)\\&-\left(n^{2k}-\frac{n!}{(n-2k)!}\right)\frac{{E_{\beta,n,k,0}}(n^{-\gamma}A,I)}{C_{\beta,n}}.
\end{align*}We first notice that\begin{align*}
0&\leq n^{2k}-\frac{n!}{(n-2k)!}=n^{2k}-\prod_{j=0}^{2k-1}(n-j)=n^{2k}-n^{2k}\prod_{j=0}^{2k-1}(1-j/n)\\&\leq n^{2k}-n^{2k}\left(1-\sum_{j=0}^{2k-1}j/n\right)=n^{2k}\sum_{j=0}^{2k-1}j/n=n^{2k-1}k(2k-1).
\end{align*}As $n^{-\gamma}A\subset(0,n^{-\gamma}c_1),\ \Sigma_{n-2,k-1,n^{-\gamma}A,I,l-1}\subset\Sigma_{n-l-1,k-l,n^{-\gamma}c_1} , $ for $n$ large enough we have $n^{1-\gamma}\beta c_1\in(0,1),$ then we infer from \eqref{22} that\begin{align*}
0&\leq {E_{\beta,n,k,0}}(n^{-\gamma}A,I)\leq {E_{\beta,n,k,0}}(n^{-\gamma}c_1)\leq C_{\beta,n-2k,k}\left(\frac{(n^{-\gamma}c_1)^{\beta+1}}{\beta+1}\right)^{k}.
\end{align*}Therefore,we have \begin{align*}
0&\leq \left(n^{2k}-\frac{n!}{(n-2k)!}\right)\frac{{E_{\beta,n,k,0}}(n^{-\gamma}A,I)}{C_{\beta,n}}\\&\leq n^{2k-1}k(2k-1) \frac{C_{\beta,n-2k,k}}{C_{\beta,n}}\left(\frac{(n^{-\gamma}c_1)^{\beta+1}}{\beta+1}\right)^{k}\\&= n^{2k-1}k(2k-1) \frac{C_{\beta,n-2k,k}}{C_{\beta,n}}\left(\frac{n^{-(\beta+2)}c_1^{\beta+1}}{\beta+1}\right)^{k}\\&= n^{-1}k(2k-1) \frac{C_{\beta,n-2k,k}}{C_{\beta,n}n^{k\beta}}\left(\frac{c_1^{\beta+1}}{\beta+1}\right)^{k}.
\end{align*}By Lemma \ref{lem12}, we have\begin{align*}&
\frac{n^{2k}}{C_{\beta,n}}\left|{E_{\beta,n,k,0}}({n^{-\gamma}A},I)-\left(\int_{n^{-\gamma}A}u^{\beta}du\right)^k
C_{\beta,n-2k,k}(I)\right|\\ \leq&\frac{n^{2k}}{C_{\beta,n}}(kn\beta c+\beta kc^2/24)\left(\frac{c^{\beta+1}}{\beta+1}\right)^{k}C_{\beta,n-2k,k}\Big|_{c=n^{-\gamma}c_1}\\ =&\frac{n^{2k}}{C_{\beta,n}}(kn^{1-\gamma}\beta c_1+\beta kn^{-2\gamma}c_1^2/24)\left(\frac{n^{-(\beta+2)}c_1^{\beta+1}}{\beta+1}\right)^{k}C_{\beta,n-2k,k}\\ =&(kn^{1-\gamma}\beta c_1+\beta kn^{-2\gamma}c_1^2/24)\left(\frac{c_1^{\beta+1}}{\beta+1}\right)^{k}\frac{C_{\beta,n-2k,k}}{C_{\beta,n}n^{k\beta}}.
\end{align*}Therefore, we have \begin{align*}&
\left|\mathbb{E}(\rho^{(k,n)}((A\times I))^k)
-\left(\int_Au^{\beta}du\right)^k\frac{C_{\beta,n-2k,k}(I)}{C_{\beta,n}n^{k\beta}}\right|\\ \leq&(kn^{1-\gamma}\beta c_1+\beta kn^{-2\gamma}c_1^2/24+n^{-1}k(2k-1) )\left(\frac{c_1^{\beta+1}}{\beta+1}\right)^{k}\frac{C_{\beta,n-2k,k}}{C_{\beta,n}n^{k\beta}}.
\end{align*}Now \eqref{26} follows from \eqref{18} of  the uniform boundedness of $\dfrac{C_{\beta,n-2k,k}}{C_{\beta,n}n^{k\beta}} $ and\begin{align*}&
\lim_{n\to+\infty}(kn^{1-\gamma}\beta c_1+\beta kn^{-2\gamma}c_1^2/24+n^{-1}k(2k-1) )=0.
\end{align*}

\section{Proof of the upper bound \eqref{19}}\label{19a}
Now we consider \eqref{19}.  We will  make use of several formulas, especially these on the generalized hypergeometric functions $_2F_1^{(\alpha)}$, where we refer   Chapter 13 of \cite{For} for more details.

By definition, we can rewrite the two-component log-gas as
\begin{align}\label{30}&{C_{\beta,n_1,2}}(I)=\int_{I^2}dr_1dr_2
|e^{ir_1}-e^{ir_2}|^{4\beta}I_{n_1,2}(\beta;r_1,r_2),\end{align}here
\begin{align*}&I_{n_1,2}(\beta;r_1,r_2):=  \int_{(-\pi,\pi)^{n_1}}d\theta_1\cdots d\theta_{n_1} \\&
\prod_{j=1}^{ n_1}\prod_{k=1}^{ 2}|1-e^{i(\theta_j-r_k)}|^{2\beta}\prod_{1\leq j<k\leq n_1}|e^{i\theta_j}-e^{i\theta_k}|^{\beta}.
\end{align*}

Now the uniform upper bound \eqref{19} is a direct consequence of the following lemma, together with the integral expression \eqref{30} (with $n_1=n-4$) and Fatou's Lemma.

\begin{lem}\label{lem6}There exists a constant $C$ depending only on $ \beta$ such that\begin{align*}&I_{n-4,2}(\beta;r_1,r_2)|e^{ir_1}-e^{ir_2}|^{4\beta}\leq CC_{\beta,n}n^{2\beta},\ \ \forall\ n>4,\ r_1,r_2\in[-\pi,\pi],
\end{align*}and\begin{align*}&\limsup_{n\to+\infty}C_{\beta,n}^{-1}n^{-2\beta}I_{n-4,2}(\beta;r_1,r_2)|e^{ir_1}-e^{ir_2}|^{4\beta}\leq (2\pi)^{-2}A_{\beta}^2.
\end{align*}\end{lem}

We need to prove several estimates in order to prove Lemma \ref{lem6}.   By Proposition 13.1.2 in \cite{For}, we have the following relation between the generalized hypergeometric function  $_2F_1^{(\alpha)}$ and the Selberg  type integrals, \begin{align}\nonumber&\frac{1}{M_{n}(a,b,1/\alpha)}\int_{-1/2}^{1/2}d\theta_1\cdots\int_{-1/2}^{1/2}d\theta_n \prod_{l=1}^n\bigg(e^{\pi i\theta_l(a-b)}|1+e^{2\pi i\theta_l}|^{a+b}\\ \nonumber&\prod_{l'=1}^m(1+t_{l'}e^{2\pi i\theta_l})\bigg)\prod_{1\leq j<k\leq n}|e^{2\pi i\theta_j}-e^{2\pi i\theta_k}|^{2/\alpha}\\ \nonumber&=_2F_1^{(1/\alpha)}(-n,\alpha b;-(n-1)-\alpha(1+a);t_1,\cdots,t_m)\\ \label{7} &=\frac{_2F_1^{(1/\alpha)}(-n,\alpha b;\alpha(a+b+m);1-t_1,\cdots,1-t_m)}{_2F_1^{(1/\alpha)}(-n,\alpha b;\alpha(a+b+m);(1)^m)},
\end{align}here, $M_n(a,b,1/\alpha)$ is defined as in \eqref{mn} and we have used the following formula (Proposition 13.1.7 in \cite{For}): \begin{align*}&_2F_1^{(\alpha)}(a, b;c;t_1,\cdots,t_m)=\\ &\frac{_2F_1^{(\alpha)}(a, b;a+b+1+(m-1)/\alpha-c;1-t_1,\cdots,1-t_m)}{_2F_1^{(\alpha)}(a, b;a+b+1+(m-1)/\alpha-c;(1)^m)}.
\end{align*}By Proposition 13.1.4 in \cite{For}, we have
\begin{align}\nonumber &\frac{1}{S_{n}(\lambda_1,\lambda_2,1/\alpha)}\int_{0}^{1}dx_1\cdots\int_{0}^{1}dx_n \prod_{l=1}^nx_l^{\lambda_1}(1-x_l)^{\lambda_2}(1-sx_l)^{-r}\\ \nonumber&\times\prod_{1\leq j<k\leq n}|{x_j}-{x_k}|^{2/\alpha}\\ \label{8} &=_2F_1^{(\alpha)}\left(r,\frac{1}{\alpha}(n-1)+\lambda_1+1;\frac{2}{\alpha}(n-1)+\lambda_1+\lambda_2+2;
(s)^n\right),
\end{align}
here, by (4.1) and (4.3) in \cite{For},   the Selberg integral is \begin{align}\label{seg}S_{n}(\lambda_1,\lambda_2,\lambda):&=\int_0^1dt_1\cdots\int_0^1dt_n\prod_{l=1}^nt_l^{\lambda_1}(1-t_l)^{\lambda_2}
\prod_{1\leq j<k\leq n}|t_j-t_k|^{2\lambda} \nonumber \\&=\prod_{j=0}^{n-1}\frac{\Gamma(\lambda_1+1+j\lambda)\Gamma(\lambda_2+1+j\lambda)
\Gamma(1+(j+1)\lambda)}{\Gamma(\lambda_1+\lambda_2+2+(n+j-1)\lambda)\Gamma(1+\lambda)}.
\end{align}

Now we change variables $ \theta_j\mapsto\theta_j+r_1\pm \pi$ to obtain\begin{align*}I_{n_1,2}(\beta;r_1,r_2)=& \int_{(-\pi,\pi)^{n_1}}d\theta_1\cdots d\theta_{n_1}
\prod_{j=1}^{ n_1}(|1+e^{i\theta_j}|^{2\beta}|1+e^{i(\theta_j+r_1-r_2)}|^{2\beta})\\ &\times\prod_{1\leq j<k\leq n_1}|e^{i\theta_j}-e^{i\theta_k}|^{\beta}.
\end{align*}For $ \beta$ positive integer, we have \begin{align*}|1+e^{i(\theta_j+r_1-r_2)}|^{2\beta}=e^{-i\beta(\theta_j+r_1-r_2)}(1+e^{i(\theta_j+r_1-r_2)})^{2\beta},
\end{align*}which shows\begin{align*}&I_{n_1,2}(\beta;r_1,r_2)= e^{-i\beta n_1(r_1-r_2)}\int_{(-\pi,\pi)^{n_1}}d\theta_1\cdots d\theta_{n_1}
\\ &\prod_{j=1}^{ n_1}\left(e^{-i\beta\theta_j}|1+e^{i\theta_j}|^{2\beta}\left(1+e^{i(\theta_j+r_1-r_2)}\right)^{2\beta}\right)\prod_{1\leq j<k\leq n_1}|e^{i\theta_j}-e^{i\theta_k}|^{\beta}.
\end{align*}Comparing with \eqref{7} and changing variables $ \theta_j\mapsto 2\pi\theta_j$, this integral is of the type therein with\begin{align*}n=n_1,\ m={2\beta},\ a-b=-{2\beta},\ a+b={2\beta},\ 2/\alpha=\beta,
\end{align*}and $$t_k=t:=e^{i(r_1-r_2)}\,\,\, \mbox{for}\,\,\,1\leq k\leq m.$$ Thus \eqref{7} shows that $I_{n_1,2} $ is proportional to\begin{align*}{t^{-\beta n_1}}_2F_1^{(\beta/2)}(-n_1,4;8;((1-t))^{2\beta}),
\end{align*} and by \eqref{mn} \eqref{7}, $_2F_1^{(\beta/2)} $ equals to 1  at the origin, thus by considering the case of $t_k=t=1$ ($1\leq k\leq 2\beta$) for $r_1=r_2 $, we will have \begin{align}\label{9}&I_{n_1,2}(\beta;r_1,r_2)=I_{n_1,2}(\beta;r_1,r_1){t^{-\beta n_1}}_2F_1^{(\beta/2)}(-n_1,4;8;((1-t))^{2\beta}),
\end{align}where \begin{align}\label{10}&I_{n_1,2}(\beta;r_1,r_1)= \int_{(-\pi,\pi)^{n_1}}d\theta_1\cdots d\theta_{n_1}
\prod_{j=1}^{ n_1}|1+e^{i\theta_j}|^{4\beta}\\ \nonumber& \times\prod_{ 1\leq j<k\leq n_1}|e^{i\theta_j}-e^{i\theta_k}|^{\beta}=(2\pi)^{n_1}M_{n_1}(2\beta,2\beta,\beta/2).
\end{align}Comparison with \eqref{8} shows that $_2F_1^{(\beta/2)} $ is of the type therein with\begin{align*}r=-n_1,\ \alpha=\beta/2,\ n=2\beta,\ \lambda_1=\lambda_2=4-\frac{1}{\alpha}(n-1)-1=\frac{2}{\beta}-1,\ s=1-t,
\end{align*}thus by \eqref{8}, we have \begin{align}\label{11}&_2F_1^{(\beta/2)}(-n_1,4;8;((1-t))^{2\beta})
=\frac{1}{S_{2\beta}({2}/{\beta}-1,{2}/{\beta}-1,{2}/{\beta})}\times
\\ \nonumber&\int_{[0,1]^{2\beta}}du_1\cdots du_{2\beta} \prod_{j=1}^{2\beta}u_j^{{2}/{\beta}-1}(1-u_j)^{{2}/{\beta}-1}(1-(1-t)u_j)^{n_1}\\ \nonumber&\times\prod_{1\leq j<k\leq 2\beta}|{u_j}-{u_k}|^{4/\beta} .
\end{align}Using \eqref{9}\eqref{10}\eqref{11}, we have\begin{align}\nonumber I_{n_1,2}(\beta;r_1,r_2)
=\frac{(2\pi)^{n_1}M_{n_1}(2\beta,2\beta,\beta/2){t^{-\beta n_1}}}{S_{2\beta}({2}/{\beta}-1,{2}/{\beta}-1,{2}/{\beta})}
\int_{[0,1]^{2\beta}}du_1\cdots du_{2\beta}\\ \label{31} \prod_{j=1}^{2\beta}u_j^{{2}/{\beta}-1}(1-u_j)^{{2}/{\beta}-1}(1-(1-t)u_j)^{n_1}\prod_{1\leq j<k\leq 2\beta}|{u_j}-{u_k}|^{4/\beta} .
\end{align}
 Now we   rewrite \eqref{31} as\begin{align*}&I_{n_1,2}(\beta;r_1,r_2)
=\frac{(2\pi)^{n_1}M_{n_1}(2\beta,2\beta,\beta/2){t^{-\beta n_1}}}{S_{2\beta}({2}/{\beta}-1,{2}/{\beta}-1,{2}/{\beta})}F_{n_1,\beta}(t),\end{align*}here $t=e^{i(r_1-r_2)}$ and we denote \begin{align*}
&F_{n_1,\beta}(t):=\int_{[0,1]^{2\beta}}du_1\cdots du_{2\beta}\\& \prod_{j=1}^{2\beta}u_j^{{2}/{\beta}-1}(1-u_j)^{{2}/{\beta}-1}(1-(1-t)u_j)^{n_1}\prod_{1\leq j<k\leq 2\beta}|{u_j}-{u_k}|^{4/\beta},
\end{align*}then $F_{n_1,\beta} $ is an analytic function (in fact a polynomial) of $t.$ As $|1-(1-t)u_j|=|1-u_j+tu_j|\leq |1-u_j|+|tu_j|=1$  for $u_j\in [0,1],\ |t|=1,$ we have\begin{align*}
|F_{n_1,\beta}(t)|&\leq\int_{[0,1]^{2\beta}}du_1\cdots du_{2\beta}\\& \prod_{j=1}^{2\beta}u_j^{{2}/{\beta}-1}(1-u_j)^{{2}/{\beta}-1}|1-(1-t)u_j|^{n_1}\prod_{1\leq j<k\leq 2\beta}|{u_j}-{u_k}|^{4/\beta}\\&\leq\int_{[0,1]^{2\beta}}du_1\cdots du_{2\beta} \prod_{j=1}^{2\beta}u_j^{{2}/{\beta}-1}(1-u_j)^{{2}/{\beta}-1}\prod_{1\leq j<k\leq 2\beta}|{u_j}-{u_k}|^{4/\beta}\\&=S_{2\beta}({2}/{\beta}-1,{2}/{\beta}-1,{2}/{\beta}),
\end{align*}which together with \eqref{33} implies\begin{align}\label{16}I_{n_1,2}(\beta;r_1,r_2)
&=\frac{(2\pi)^{n_1}M_{n_1}(2\beta,2\beta,\beta/2)}{S_{2\beta}({2}/{\beta}-1,{2}/{\beta}-1,{2}/{\beta})}|F_{n_1,\beta}(t)|\\ \nonumber&\leq (2\pi)^{n_1}M_{n_1}(2\beta,2\beta,\beta/2)= (2\pi)^{-1}C_{\beta,n_1,(4)}.\end{align} Changing variables $ u_j\mapsto t_j/(1+t_j)$, we obtain\begin{align*}
F_{n_1,\beta}(t)&=\int_{(0,+\infty)^{2\beta}}\frac{dt_1\cdots dt_{2\beta}}{(1+t_1)^2\cdots(1+t_{2\beta})^2}\\& \prod_{j=1}^{2\beta}\frac{t_j^{{2}/{\beta}-1}}{(1+t_j)^{2({2}/{\beta}-1)}}\left(\frac{1+tt_j}{1+t_j}\right)^{n_1}\prod_{1\leq j<k\leq 2\beta}\left|\frac{{t_j}-{t_k}}{(1+t_j)(1+t_k)}\right|^{4/\beta}\\& =\int_{(0,+\infty)^{2\beta}}{dt_1\cdots dt_{2\beta}}\prod_{j=1}^{2\beta}\frac{t_j^{{2}/{\beta}-1}(1+tt_j)^{n_1}}{(1+t_j)^{2({2}/{\beta}-1)+2+n_1+4/\beta\cdot(2\beta-1)}}
\\&\times\prod_{1\leq j<k\leq 2\beta}\left|{t_j}-{t_k}\right|^{4/\beta}.
\end{align*}Since $2({2}/{\beta}-1)+2+4/\beta\cdot(2\beta-1)={4}/{\beta}+8-{4}/{\beta}=8, $ we have\begin{align*}
F_{n_1,\beta}(-z^2)=&\int_{(0,+\infty)^{2\beta}}{dt_1\cdots dt_{2\beta}}\prod_{j=1}^{2\beta}\frac{t_j^{{2}/{\beta}-1}(1-z^2t_j)^{n_1}}{(1+t_j)^{8+n_1}}
\\&\times\prod_{1\leq j<k\leq 2\beta}\left|{t_j}-{t_k}\right|^{4/\beta}.
\end{align*}For $z\in (0,+\infty),$ a simple changing of variables $zt_j\mapsto s_j$ shows that\begin{align*}
&F_{n_1,\beta}(-z^2)=z^{-8\beta}\int_{(0,+\infty)^{2\beta}}{ds_1\cdots ds_{2\beta}}\prod_{j=1}^{2\beta}\frac{s_j^{{2}/{\beta}-1}(1-zs_j)^{n_1}}{(1+z^{-1}s_j)^{8+n_1}}
\\&\times\prod_{1\leq j<k\leq 2\beta}\left|{s_j}-{s_k}\right|^{4/\beta}.
\end{align*}Since both sides are analytic functions of $z$ for $\text{Re}z>0,$ this identity is always true for $\text{Re}z>0,$ moreover,  we can decompose $(0,+\infty)$ into $(0,1]\cup[1,+\infty)$ and use the symmetry of $s_j$ to obtain\begin{align}
\label{13}&F_{n_1,\beta}(-z^2)=z^{-8\beta}\sum_{l=0}^{2\beta}{2\beta\choose l}F_{n_1,\beta,l}(z),\ \ \ \ \text{Re}z>0,\end{align}where\begin{align*}F_{n_1,\beta,l}(z):=&\int_{(0,1]^l\times[1,+\infty)^{2\beta-l}}{ds_1\cdots ds_{2\beta}}\prod_{j=1}^{2\beta}\frac{s_j^{{2}/{\beta}-1}(1-zs_j)^{n_1}}{(1+z^{-1}s_j)^{8+n_1}}
\\ &\times\prod_{1\leq j<k\leq 2\beta}\left|{s_j}-{s_k}\right|^{4/\beta}.
\end{align*}The changing of variables $s_j\mapsto s_j^{-1}$ for $l<j\leq2\beta$ shows that\begin{align*}F_{n_1,\beta,l}(z)&=\int_{(0,1]^{2\beta}}{ds_1\cdots ds_{2\beta}}\prod_{j=1}^{l}\frac{s_j^{{2}/{\beta}-1}(1-zs_j)^{n_1}}{(1+z^{-1}s_j)^{8+n_1}}\times\\&
\prod_{j=l+1}^{2\beta}\frac{s_j^{-{2}/{\beta}+1}(1-zs_j^{-1})^{n_1}}{(1+z^{-1}s_j^{-1})^{8+n_1}s_j^2}
\prod_{1\leq j<k\leq l}\left|{s_j}-{s_k}\right|^{4/\beta}\prod_{l< j<k\leq 2\beta}\left|s_j^{-1}-s_k^{-1}\right|^{4/\beta}\\&\times\prod_{j=1}^{l}\prod_{k=l+1}^{2\beta}\left|s_j-s_k^{-1}\right|^{4/\beta}
\\&=\int_{(0,1]^{2\beta}}{ds_1\cdots ds_{2\beta}}\prod_{j=1}^{l}\frac{s_j^{{2}/{\beta}-1}(1-zs_j)^{n_1}}{(1+z^{-1}s_j)^{8+n_1}}
\prod_{j=l+1}^{2\beta}\frac{s_j^{a}(s_j-z)^{n_1}}{(s_j+z^{-1})^{8+n_1}}\\&
\times\prod_{1\leq j<k\leq l}\left|{s_j}-{s_k}\right|^{4/\beta}\prod_{l< j<k\leq 2\beta}\left|s_j-s_k\right|^{4/\beta}\prod_{j=1}^{l}\prod_{k=l+1}^{2\beta}\left|1-s_js_k\right|^{4/\beta},
\end{align*}here, $a=-{2}/{\beta}+1+8-2-4/\beta\cdot(2\beta-1)={2}/{\beta}-1$. For $z=e^{i\theta},\ \theta\in (-\pi/2,\pi/2)$ i.e., $\text{Re} z>0$, and for  $s>0$, we have $ |1+z^{-1}s|^2=|s+z^{-1}|^2=1+s^2+2s\cos\theta>1$ and $|1-zs|=|s-z|$, therefore, we have \begin{align}\label{14}|F_{n_1,\beta,l}(e^{i\theta})|&\leq\int_{(0,1]^{2\beta}}{ds_1\cdots ds_{2\beta}}\prod_{j=1}^{2\beta}\frac{s_j^{{2}/{\beta}-1}|1-e^{i\theta}s_j|^{n_1}}{|1+e^{-i\theta}s_j|^{n_1}}\times\nonumber\\ &
\prod_{1\leq j<k\leq l}\left|{s_j}-{s_k}\right|^{4/\beta}\prod_{l< j<k\leq 2\beta}\left|s_j-s_k\right|^{4/\beta}\nonumber\\&=F_{n_1,\beta,(l)}(\theta)F_{n_1,\beta,(2\beta-l)}(\theta),
\end{align} here, we used $\left|1-s_js_k\right|\leq 1$ and we denote \begin{align*}F_{n_1,\beta,(l)}(\theta):=&\int_{(0,1]^l}{ds_1\cdots ds_{l}}\prod_{j=1}^{l}\frac{s_j^{{2}/{\beta}-1}|1-e^{i\theta}s_j|^{n_1}}{|1+e^{-i\theta}s_j|^{n_1}}
\prod_{1\leq j<k\leq l}\left|{s_j}-{s_k}\right|^{4/\beta}.
\end{align*}As $\dfrac{1-s}{1+s}\leq e^{-2s}$ for $s\in(0,1)$, we have\begin{align*}\frac{|1-e^{i\theta}s|^{n_1}}{|1+e^{-i\theta}s|^{n_1}}
=\left|\frac{1+s^2-2s\cos\theta}{1+s^2+2s\cos\theta}\right|^{n_1/2}\leq e^{-\frac{2sn_1\cos\theta}{1+s^2}},
\end{align*}which implies\begin{align*} F_{n_1,\beta,(l)}(\theta)\leq&\int_{(0,1]^l}{ds_1\cdots ds_{l}}\prod_{j=1}^{l}s_j^{{2}/{\beta}-1}e^{-\frac{2s_jn_1\cos\theta}{1+s_j^2}}
\prod_{1\leq j<k\leq l}\left|{s_j}-{s_k}\right|^{4/\beta}\\ \leq&\int_{(0,1]^l}{ds_1\cdots ds_{l}}\prod_{j=1}^{l}s_j^{{2}/{\beta}-1}e^{-s_jn_1\cos\theta}
\prod_{1\leq j<k\leq l}\left|{s_j}-{s_k}\right|^{4/\beta}.\end{align*}
We denote\begin{align*}&J_{n,\beta}(z):=\int_{(0,+\infty)^n}\prod_{j=1}^{n}t_j^{{2}/{\beta}-1}e^{-zt_j}\prod_{1\leq j<k\leq n}|{t_j}-{t_k}|^{4/\beta}dt_1\cdots dt_{n},
\end{align*}then we have $$J_{n,\beta}(z)=z^{-2n^2/\beta}J_{n,\beta}(1).$$ According to Proposition 4.7.3 in \cite{For}, we have the explicit
evaluation\begin{align}\label{12}&J_{n,\beta}(1)=\prod_{j=1}^{n}\frac{\Gamma(1+2j/\beta)\Gamma(2j/\beta)}{\Gamma(1+2/\beta)}.
\end{align}
By the definition of $J_{n,\beta} $, we first easily have  the upper bound \begin{align}\label{15}F_{n_1,\beta,(l)}(\theta)\leq& J_{l,\beta}(n_1\cos\theta)=(n_1\cos\theta)^{-2l^2/\beta}J_{l,\beta}(1).
\end{align}We change of variables $n_1s_j\mapsto t_j$ to get \begin{align*} &F_{n_1,\beta,(l)}(\theta)\leq n_1^{-2l^2/\beta}\int_{(0,n_1]^l}{dt_1\cdots dt_{l}}\\&\prod_{j=1}^{l}t_j^{{2}/{\beta}-1}e^{-\frac{2t_j\cos\theta}{1+t_j^2n_1^{-2}}}
\prod_{1\leq j<k\leq l}\left|{t_j}-{t_k}\right|^{4/\beta}.\end{align*}By  the dominated convergence theorem, we further have\begin{align*} &\limsup_{n_1\to+\infty}n_1^{2l^2/\beta}F_{n_1,\beta,(l)}(\theta)\leq \int_{(0,+\infty)^l}{dt_1\cdots dt_{l}}\prod_{j=1}^{l}t_j^{{2}/{\beta}-1}e^{-{2t_j\cos\theta}}\\&\times
\prod_{1\leq j<k\leq l}\left|{t_j}-{t_k}\right|^{4/\beta}=J_{l,\beta}(2\cos\theta)=(2\cos\theta)^{-2l^2/\beta}J_{l,\beta}(1).\end{align*}Therefore, we have\begin{align} \label{17} &\limsup_{n_1\to+\infty}(2n_1\cos\theta)^{2l^2/\beta}F_{n_1,\beta,(l)}(\theta)\leq J_{l,\beta}(1).\end{align}
\subsection{Proof of Lemma  \ref{lem6}}
Now we are ready to give the proof of Lemma \ref{lem6}.\begin{proof} If $ |e^{ir_1}-e^{ir_2}|\leq n^{-1}$,  then the first inequality holds by \eqref{16} with $n_1=n-4$ and Lemma \ref{lem7}, i.e.,  \begin{align*}&I_{n-4,2}(\beta;r_1,r_2)|e^{ir_1}-e^{ir_2}|^{4\beta}\leq (2\pi)^{-1}C_{\beta,n-4,(4)}n^{-4\beta}\leq CC_{\beta,n}n^{2\beta}.
\end{align*}If $ |e^{ir_1}-e^{ir_2}|\geq n^{-1},$ as $t=e^{i(r_1-r_2)}$, we have $|t-1|=|e^{ir_1}-e^{ir_2}|\geq n^{-1}$ and we can write $t=-e^{2i\theta}$ for some $ \theta\in (-\pi/2,\pi/2),$ then by \eqref{33} and \eqref{16},  we have\begin{align*}&I_{n-4,2}(\beta;r_1,r_2)|e^{ir_1}-e^{ir_2}|^{4\beta}\\=&
\frac{(2\pi)^{n_1}M_{n-4}(2\beta,2\beta,\beta/2)}{S_{2\beta}({2}/{\beta}-1,{2}/{\beta}-1,{2}/{\beta})}
|F_{n-4,\beta}(t)||1-t|^{4\beta}\\=& \frac{(2\pi)^{-1}C_{\beta,n-4,(4)}|1-t|^{4\beta}}{S_{2\beta}({2}/{\beta}-1,{2}/{\beta}-1,{2}/{\beta})}
|F_{n-4,\beta}(t)|.\end{align*}By \eqref{13} and \eqref{14}, we have\begin{align*} |F_{n-4,\beta}(t)|\leq&
\sum_{l=0}^{2\beta}{2\beta\choose l}|F_{n-4,\beta,l}(e^{i\theta})|\\ \leq&
\sum_{l=0}^{2\beta}{2\beta\choose l}F_{n-4,\beta,(l)}({\theta})F_{n-4,\beta,(2\beta-l)}({\theta}),\end{align*}thus we have
\begin{align*}I_{n-4,2}(\beta;r_1,r_2)|e^{ir_1}-e^{ir_2}|^{4\beta}\leq\sum_{l=0}^{2\beta}I_{n-4,2}^{(l)}(\beta;r_1,r_2),\end{align*}
where
\begin{align*}
&I_{n-4,2}^{(l)}(\beta;r_1,r_2)=\frac{(2\pi)^{-1}C_{\beta,n-4,(4)}|1-t|^{4\beta}}{S_{2\beta}({2}/{\beta}-1,{2}/{\beta}-1,{2}/{\beta})}
{2\beta\choose l}F_{n-4,\beta,(l)}({\theta})F_{n-4,\beta,(2\beta-l)}({\theta}).
\end{align*}As $t=-e^{2i\theta},$ we know that $|1-t|=2\cos\theta\geq n^{-1}, $ by \eqref{15} and Lemma \ref{lem7} we have \begin{align*}
I_{n-4,2}^{(l)}(\beta;r_1,r_2)&\leq \frac{CC_{\beta,n}n^{6\beta}(2\cos\theta)^{4\beta}}{S_{2\beta}({2}/{\beta}-1,{2}/{\beta}-1,{2}/{\beta})}{2\beta\choose l}\times\\&(n_1\cos\theta)^{-2l^2/\beta}J_{l,\beta}(1)(n_1\cos\theta)^{-2(2\beta-l)^2/\beta}J_{2\beta-l,\beta}(1)\\ \leq& {CC_{\beta,n}n^{2\beta}(2n\cos\theta)^{4\beta}}(n_1\cos\theta)^{-2l^2/\beta-2(2\beta-l)^2/\beta}\\ \leq& {CC_{\beta,n}n^{2\beta}(n_1\cos\theta)^{4\beta}}(n_1\cos\theta)^{-4(\beta^2+(\beta-l)^2)/\beta}\\ =& CC_{\beta,n}n^{2\beta}(n_1\cos\theta)^{-4(\beta-l)^2/\beta}\leq CC_{\beta,n}n^{2\beta},
\end{align*}here $n_1=n-4,$ $n_1\cos\theta=n_1|1-t|/2\geq n_1/(2n)\geq 1/10, $ and $C$ is a constant depending only on $ \beta,l.$ Summing up, we will conclude the first inequality.

 Now we consider the second inequality regarding the limit superior. If $ |e^{ir_1}-e^{ir_2}|=0$, then
the result is clearly true. If $ |e^{ir_1}-e^{ir_2}|>0,$ then we can write $t=e^{i(r_1-r_2)}=-e^{2i\theta}$ for some $ \theta\in (-\pi/2,\pi/2),$ and $|1-t|=2\cos\theta.$ Recall that \begin{align*}
&0\leq I_{n-4,2}^{(l)}(\beta;r_1,r_2)\leq  CC_{\beta,n}n^{2\beta}(n_1\cos\theta)^{-4(\beta-l)^2/\beta},\ n_1=n-4,
\end{align*}then for $l\neq\beta$, we have\begin{align*}&\lim_{n\to+\infty}C_{\beta,n}^{-1}n^{-2\beta}I_{n-4,2}^{(l)}(\beta;r_1,r_2)=0,
\end{align*}thus\begin{align}\label{39}&\limsup_{n\to+\infty}C_{\beta,n}^{-1}n^{-2\beta}I_{n-4,2}(\beta;r_1,r_2)|e^{ir_1}-e^{ir_2}|^{4\beta}
\\ \nonumber&\leq \limsup_{n\to+\infty}C_{\beta,n}^{-1}n^{-2\beta}I_{n-4,2}^{(\beta)}(\beta;r_1,r_2).
\end{align}
Notice that\begin{align*}
&I_{n-4,2}^{(\beta)}(\beta;r_1,r_2)=\frac{(2\pi)^{-1}C_{\beta,n-4,(4)}|1-t|^{4\beta}}
{S_{2\beta}({2}/{\beta}-1,{2}/{\beta}-1,{2}/{\beta})}
{2\beta\choose \beta}|F_{n-4,\beta,(\beta)}({\theta})|^2,
\end{align*}we have\begin{align*}
&C_{\beta,n}^{-1}n^{-2\beta}I_{n-4,2}^{(\beta)}(\beta;r_1,r_2)\\=&
\frac{(2\pi)^{-1}C_{\beta,n}^{-1}n^{-2\beta}C_{\beta,n-4,(4)}(2\cos\theta)^{4\beta}}
{S_{2\beta}({2}/{\beta}-1,{2}/{\beta}-1,{2}/{\beta})}{2\beta\choose \beta}|F_{n-4,\beta,(\beta)}({\theta})|^2\\=&
\frac{C_{\beta,n-4,(4)}}{C_{\beta,n}n^{6\beta}}\frac{(2\pi)^{-1}(2n\cos\theta)^{4\beta}}
{S_{2\beta}({2}/{\beta}-1,{2}/{\beta}-1,{2}/{\beta})}{2\beta\choose \beta}|F_{n-4,\beta,(\beta)}({\theta})|^2.
\end{align*}Therefore, by \eqref{17}, Lemma \ref{lem7} and Lemma \ref{lem5} below, we have\begin{align*}&\limsup_{n\to+\infty}C_{\beta,n}^{-1}n^{-2\beta}I_{n-4,2}^{(\beta)}(\beta;r_1,r_2)
=\lim_{n\to+\infty}\frac{C_{\beta,n-4,(4)}}{C_{\beta,n}n^{6\beta}}{2\beta\choose \beta}\times\\&\frac{(2\pi)^{-1}}
{S_{2\beta}({2}/{\beta}-1,{2}/{\beta}-1,{2}/{\beta})}
\left|\limsup_{n\to+\infty}(2n\cos\theta)^{2\beta}F_{n-4,\beta,(\beta)}(\theta)\right|^2\\ \leq&\frac{(2\pi)^{-1}A_{\beta,4}}
{S_{2\beta}({2}/{\beta}-1,{2}/{\beta}-1,{2}/{\beta})}{2\beta\choose \beta}
\left|J_{\beta,\beta}(1)\right|^2=(2\pi)^{-2}A_{\beta}^2.
\end{align*}  This, together with \eqref{39}, will complete the proof of Lemma \ref{lem6} provided Lemma \ref{lem5}.\end{proof}

Now we prove the following identity to complete Lemma  \ref{lem6}.
\begin{lem}\label{lem5}It holds that\begin{align*}&(2\pi)A_{\beta,4}{2\beta\choose \beta}\frac{|J_{\beta,\beta}(1)|^2}{S_{2\beta}({2}/{\beta}-1,{2}/{\beta}-1,{2}/{\beta})}
=A_{\beta}^2.
\end{align*}\end{lem}\begin{proof}Notice that the Selberg integral \begin{align*}&S_{2\beta}({2}/{\beta}-1,{2}/{\beta}-1,{2}/{\beta})=\prod_{j=0}^{2\beta-1}\frac{(\Gamma(2(j+1)/{\beta}))^2
\Gamma(1+2(j+1)/{\beta})}{\Gamma(2(2\beta+j+1)/{\beta})\Gamma(1+{2}/{\beta})}\\&=\prod_{j=1}^{2\beta}\frac{(\Gamma(2j/{\beta}))^2
\Gamma(1+2j/{\beta})}{\Gamma(2j/{\beta}+4)\Gamma(1+{2}/{\beta})}=\prod_{j=1}^{2\beta}\frac{(\Gamma(2j/{\beta}))^2
}{\prod_{k=1}^3(2j/{\beta}+k)\Gamma(1+{2}/{\beta})},\end{align*}that
\begin{align*}\prod_{j=1}^{2\beta}\prod_{k=1}^3(2j/{\beta}+k)&=(2/{\beta})^{6\beta}\prod_{k=1}^3\prod_{j=1}^{2\beta}(j+k{\beta}/2)
=(2/{\beta})^{6\beta}\prod_{k=1}^3\frac{\Gamma(1+(k+4){\beta}/2)}{\Gamma(1+k{\beta}/2)},
\end{align*}and that (using \eqref{12})\begin{align*}&\prod_{j=1}^{2\beta}\frac{(\Gamma(2j/{\beta}))^2
}{\Gamma(1+{2}/{\beta})}=\prod_{j=1}^{\beta}\prod_{k=0}^{1}\frac{(\Gamma(2(j+k\beta)/{\beta}))^2
}{\Gamma(1+{2}/{\beta})}\\=&\prod_{j=1}^{\beta}\frac{(\Gamma(2j/{\beta})\Gamma(2j/{\beta}+2))^2
}{(\Gamma(1+{2}/{\beta}))^2}=|J_{\beta,\beta}(1)|^2\prod_{j=1}^{\beta}(2j/{\beta}+1)^2\\
=&|J_{\beta,\beta}(1)|^2(2/{\beta})^{2\beta}\frac{(\Gamma(1+3{\beta}/2))^2}{(\Gamma(1+{\beta}/2))^2},\end{align*}we have\begin{align*}&(2\pi)A_{\beta,4}{2\beta\choose \beta}\frac{|J_{\beta,\beta}(1)|^2}{S_{2\beta}({2}/{\beta}-1,{2}/{\beta}-1,{2}/{\beta})}\\
=&(2\pi)A_{\beta,4}\frac{\Gamma(1+2{\beta})}{(\Gamma(1+{\beta}))^2}(2/{\beta})^{4\beta}
\frac{(\Gamma(1+{\beta}/2))^2}{(\Gamma(1+3{\beta}/2))^2}
\prod_{k=1}^3\frac{\Gamma(1+(k+4){\beta}/2)}{\Gamma(1+k{\beta}/2)},\end{align*}as in Lemma \ref{lem7},
\begin{align*}A_{\beta,4}=\frac{(2\pi)^{-3}(\Gamma(\beta/2+1))^{4}}
{\Gamma(2\beta+1)}(\beta/2)^{6\beta}\prod_{j=1}^{3}
\frac{\Gamma(j\beta/2 +1)}{\Gamma((4+j)\beta/2 +1)},
\end{align*}we have\begin{align*}&(2\pi)A_{\beta,4}{2\beta\choose \beta}\frac{|J_{\beta,\beta}(1)|^2}{S_{2\beta}({2}/{\beta}-1,{2}/{\beta}-1,{2}/{\beta})}\\
=&\frac{(2\pi)^{-2}(\Gamma(\beta/2+1))^{4}}
{\Gamma(2\beta+1)}(\beta/2)^{2\beta}\frac{\Gamma(1+2{\beta})}{(\Gamma(1+{\beta}))^2}
\frac{(\Gamma(1+{\beta}/2))^2}{(\Gamma(1+3{\beta}/2))^2}\\
=&\frac{(2\pi)^{-2}(\Gamma(\beta/2+1))^{6}}
{(\Gamma(1+{\beta}))^2(\Gamma(1+3{\beta}/2))^2}(\beta/2)^{2\beta}=A_{\beta}^2
,\end{align*}this completes the proof.\end{proof}

\section{Proof of Lemma \ref{lem9}}\label{ccccc}Now we give the proof of Lemma \ref{lem9}.

\begin{proof}As $C_{\beta,n-2,1}(I)=|I| C_{\beta,n-2,1}/(2\pi)$ (recall \eqref{32}),  by Lemma \ref{lem7}, we have\begin{align}\label{k1}&\lim_{n\to+\infty}\frac{C_{\beta,n-2,1}(I)}{C_{\beta,n}n^{\beta}}
=\frac{|I|}{2\pi}\lim_{n\to+\infty}\frac{C_{\beta,n-2,1}}{C_{\beta,n}n^{\beta}}
=\frac{|I|A_{\beta}}{2\pi},
\end{align}i.e., Lemma \ref{lem9} is true for $k=1.$ Now we assume $|I|>0,$ then for every $ \lambda>0$, we can find $A=(0,a(\lambda))$ such that\begin{align*}&\lambda
=\int_Au^{\beta}du\times\frac{|I|A_{\beta}}{2\pi}.
\end{align*}We denote $$X_n:= \widetilde{\chi}^{(n)}(A\times I),$$ then by Lemma \ref{lem8} with $k=1$ and \eqref{k1}, we have\begin{align*}&
\lim_{n\to+\infty}\mathbb{E} X_n
=\lim_{n\to+\infty}\left(\int_Au^{\beta}du\right)\frac{C_{\beta,n-2,1}(I)}{C_{\beta,n}n^{\beta}}=\lambda;
\end{align*}and with $k=2$ in Lemma \ref{lem8}, we have\begin{align*}&
\liminf_{n\to+\infty}\mathbb{E}( X_n(X_n-1))
=\liminf_{n\to+\infty}\left(\int_Au^{\beta}du\right)^2\frac{C_{\beta,n-4,2}(I)}{C_{\beta,n}n^{2\beta}}.
\end{align*}On the other hand, by H\"older inequality, we have $\mathbb{E}( X_n)^2\geq (\mathbb{E} X_n)^2$ and $\mathbb{E}( X_n(X_n-1))\geq (\mathbb{E} X_n)^2-(\mathbb{E} X_n)$, and thus we have  \begin{align*}&
\liminf_{n\to+\infty}\mathbb{E}( X_n(X_n-1))
\geq \liminf_{n\to+\infty}((\mathbb{E} X_n)^2-(\mathbb{E} X_n))=\lambda^2-\lambda.
\end{align*}Therefore,  we have\begin{align*}&
\liminf_{n\to+\infty}\frac{C_{\beta,n-4,2}(I)}{C_{\beta,n}n^{2\beta}}\geq\left(\int_Au^{\beta}du\right)^{-2}(\lambda^2-\lambda)
=(1-\lambda^{-1})\left(\frac{|I|A_{\beta}}{2\pi}\right)^2.
\end{align*}Letting $ \lambda\to+\infty$, we have\begin{align*}&
\liminf_{n\to+\infty}\frac{C_{\beta,n-4,2}(I)}{C_{\beta,n}n^{2\beta}}\geq\left(\frac{|I|A_{\beta}}{2\pi}\right)^2,
\end{align*}which along with \eqref{19} gives Lemma \ref{lem9} for $k=2.$

 Moreover,  since \begin{align*}&
\mathbb{E}( X_n-\lambda)^2=\mathbb{E}( X_n(X_n-1))-(2\lambda-1)(\mathbb{E} X_n)+\lambda^2,
\end{align*}by Lemma \ref{lem8} and \eqref{19}, we have\begin{align*}
\limsup_{n\to+\infty}\mathbb{E}( X_n(X_n-1))
&=\limsup_{n\to+\infty}\left(\int_Au^{\beta}du\right)^2\frac{C_{\beta,n-4,2}(I)}{C_{\beta,n}n^{2\beta}}
\\&\leq\left(\int_Au^{\beta}du\right)^2\left(\frac{|I|A_{\beta}}{2\pi}\right)^2=\lambda^2,
\end{align*}and thus we have  \begin{align*}
\limsup_{n\to+\infty}\mathbb{E}( X_n-\lambda)^2\leq \lambda^2-(2\lambda-1)\lambda+\lambda^2=\lambda.
\end{align*}Now we denote by $C$ a constant independent
of $n,\lambda$, which may be different from line to line. As $X_n^k\leq \dfrac{2X_n!}{(X_n-k)!}+C$ ($-C$ can be chosen as the lower bound of the polynomial
$2x(x-1)\cdots (x-k+1)-x^k$ for $x\geq0$), by Lemma \ref{lem8} and \eqref{18}, we have\begin{align*}
\limsup_{n\to+\infty}\mathbb{E}( X_n^k)&\leq
2\limsup_{n\to+\infty}\mathbb{E}\left(\frac{X_n!}{(X_n-k)!}\right)+C
\\&\leq2\limsup_{n\to+\infty}\left(\int_Au^{\beta}du\right)^k\frac{C_{\beta,n-2k,k}(I)}{C_{\beta,n}n^{k\beta}}+C
\\&\leq2\left(\int_Au^{\beta}du\right)^k\limsup_{n\to+\infty}\frac{C_{\beta,n-2k,k}}{C_{\beta,n}n^{k\beta}}+C
\\&\leq C\left(\int_Au^{\beta}du\right)^k+C\leq C\lambda^k+C.
\end{align*}By H\"older inequality, we have\begin{align*}
\mathbb{E}\left(\dfrac{( X_n-\lambda)^2X_n!}{(X_n-k+1)!}\right)
&\leq \mathbb{E}\left(( X_n-\lambda)^2X_n^{k-1}\right)\\&\leq\left(\mathbb{E}( X_n-\lambda)^2\right)^{\frac{1}{2}}\left(\mathbb{E}\left(( X_n-\lambda)^2X_n^{2k-2}\right)\right)^{\frac{1}{2}}\\&\leq\left(\mathbb{E}( X_n-\lambda)^2\right)^{\frac{1}{2}}\left(\mathbb{E}\left( X_n^{2k}+\lambda^2X_n^{2k-2}\right)\right)^{\frac{1}{2}},
\end{align*} and thus for any positive integer $k$, we have \begin{align}\label{upperd}\nonumber &
\limsup_{n\to+\infty}\mathbb{E}\left(\dfrac{( X_n-\lambda)^2X_n!}{(X_n-k+1)!}\right)
\\ \leq\nonumber&\left(\limsup_{n\to+\infty}\mathbb{E}( X_n-\lambda)^2\right)^{\frac{1}{2}}\left(\limsup_{n\to+\infty}\mathbb{E}\left( X_n^{2k}+\lambda^2X_n^{2k-2}\right)\right)^{\frac{1}{2}}\\ \leq &\lambda^{\frac{1}{2}}\left((C\lambda^{2k}+C)+\lambda^2(C\lambda^{2k-2}+C)\right)^{\frac{1}{2}}\leq C\lambda^{\frac{1}{2}}(\lambda^{k}+1).
\end{align}Now we can prove the result by induction. Assume $j\geq 2$ and Lemma \ref{lem9} is true for $k=j,j-1$, then by Lemma \ref{lem8}, we further have\begin{align*}
\lim_{n\to+\infty}\mathbb{E}\left(\frac{X_n!}{(X_n-k)!}\right)
&=\lim_{n\to+\infty}\left(\int_Au^{\beta}du\right)^k\frac{C_{\beta,n-2k,k}(I)}{C_{\beta,n}n^{k\beta}}
\\&=\left(\int_Au^{\beta}du\right)^k\left(\frac{|I|A_{\beta}}{2\pi}\right)^k=\lambda^k,\ \ k=j-1,\,j.
\end{align*}
We note that  $( X_n-\lambda)^2=( X_n-k)(X_n-k-1)-(2\lambda-2k-1)(X_n-k)+(\lambda-k)^2 ,$ then for any integer $k\geq 2$, we have the identity\begin{align}\label{iden}&
\dfrac{( X_n-\lambda)^2X_n!}{(X_n-k)!}=\dfrac{X_n!}{(X_n-k-2)!}-\dfrac{(2\lambda-2k-1)X_n!}{(X_n-k-1)!}
+\dfrac{(\lambda-k)^2X_n!}{(X_n-k)!}.
\end{align}

Now by induction, \eqref{upperd}\eqref{iden} and Lemma \ref{lem8}, we have\begin{align*}
C\lambda^{\frac{1}{2}}(\lambda^{j}+1)\geq&\limsup_{n\to+\infty}\mathbb{E}\left(\dfrac{( X_n-\lambda)^2X_n!}{(X_n-j+1)!}\right)
\\ =&\limsup_{n\to+\infty}\mathbb{E}\left(\dfrac{X_n!}{(X_n-j-1)!}-\dfrac{(2\lambda-2j+1)X_n!}{(X_n-j)!}
+\dfrac{(\lambda-j+1)^2X_n!}{(X_n-j+1)!}\right)\\ =&\limsup_{n\to+\infty}\mathbb{E}\left(\dfrac{X_n!}{(X_n-j-1)!}\right)-(2\lambda-2j+1)\lambda^j
+(\lambda-j+1)^2\lambda^{j-1}\\ =&\limsup_{n\to+\infty}\left(\int_Au^{\beta}du\right)^k\frac{C_{\beta,n-2k,k}(I)}{C_{\beta,n}n^{k\beta}}-(\lambda^2-(j-1)^2-\lambda)
\lambda^{j-1},
\end{align*}where we denote $k=j+1$ in the last line.  Therefore, as $\lambda$ large enough, we have\begin{align*}
\limsup_{n\to+\infty}\frac{C_{\beta,n-2k,k}(I)}{C_{\beta,n}n^{k\beta}} &\leq\left(\int_Au^{\beta}du\right)^{-k}(
\lambda^{j+1}+C\lambda^{\frac{1}{2}}(\lambda^{j}+1))\\ &=\left(\frac{|I|A_{\beta}}{2\pi}\right)^k(
1+C\lambda^{-\frac{1}{2}}+C\lambda^{-j-\frac{1}{2}}).
\end{align*}Letting $ \lambda\to+\infty$, we have\begin{align*}&
\limsup_{n\to+\infty}\frac{C_{\beta,n-2k,k}(I)}{C_{\beta,n}n^{k\beta}} \leq\left(\frac{|I|A_{\beta}}{2\pi}\right)^k.
\end{align*}Similarly, as $ \dfrac{( X_n-\lambda)^2X_n!}{(X_n-j+1)!}\geq 0$,  by induction and Lemma \ref{lem8} again,  we have\begin{align*}0\leq&\liminf_{n\to+\infty}\mathbb{E}\left(\dfrac{X_n!}{(X_n-j-1)!}-\dfrac{(2\lambda-2j+1)X_n!}{(X_n-j)!}
+\dfrac{(\lambda-j+1)^2X_n!}{(X_n-j+1)!}\right)\\ =&\liminf_{n\to+\infty}\left(\int_Au^{\beta}du\right)^k\frac{C_{\beta,n-2k,k}(I)}{C_{\beta,n}n^{k\beta}}-(\lambda^2-(j-1)^2-\lambda)
\lambda^{j-1},
\end{align*}where $k=j+1$ again. Therefore, we have\begin{align*}
\liminf_{n\to+\infty}\frac{C_{\beta,n-2k,k}(I)}{C_{\beta,n}n^{k\beta}} &\geq\left(\int_Au^{\beta}du\right)^{-k}(\lambda^2-(j-1)^2-\lambda)
\lambda^{j-1}\\&=\left(\frac{|I|A_{\beta}}{2\pi}\right)^k(
1-\lambda^{-1}-(j-1)^2\lambda^{-2}).
\end{align*}Letting $ \lambda\to+\infty$ again, we have\begin{align*}&
\liminf_{n\to+\infty}\frac{C_{\beta,n-2k,k}(I)}{C_{\beta,n}n^{k\beta}} \geq\left(\frac{|I|A_{\beta}}{2\pi}\right)^k,
\end{align*}thus Lemma \ref{lem9} is also true for $k=j+1.$ This completes the proof.\end{proof}


\begin{thebibliography}{99}
\bibitem{A}K. Aomoto, The complex Selberg integral, Quart. J. Math. Oxford 38 (1987), 385-399.
\bibitem{A1}K. Aomoto, Jacobi polynomials associated with Selberg’s integral, SIAM J. Math. Analysis 18 (1987), 545-549.
\bibitem{AGZ}G. W. Anderson,  A. Guionnet and O. Zeitouni, \emph{An introduction to random matrices}. Cambridge Studies in Advanced Mathematics, 118. Cambridge University Press, Cambridge, 2010.
\bibitem{BB}G. Ben Arous and P. Bourgade, Extreme gaps between eigenvalues of random matrices. Ann. Prob. 41, 2648-2681 (2013).
\bibitem{B}P.  Bourgade, Extreme gaps between eigenvalues of Wigner matrices, arXiv:1812.10376.
\bibitem{FTW} R. Feng, G. Tian and D. Wei, Small gaps of GOE, Geom. Funct. Anal. 29, 1794-1827 (2019). 
\bibitem{FW2} R. Feng and D. Wei, Large gaps of CUE and GUE, arXiv:1807.02149.
\bibitem{FG}A. Figalli and A. Guionnet,
Universality in several-matrix models via approximate transport maps, Acta Math.
Volume 217, Number 1 (2016), 81-176.
    \bibitem{For}P.J. Forrester, \emph{Log-gases and random matrices}, LMS-34, Princeton University Press, 2010.
\bibitem{ds}B. Landon,  P. Lopatto and J. Marcinek, Comparison theorem for some extremal eigenvalue statistics, arXiv:1812.10022. 

\bibitem{STKZ}M. Smaczynski, T. Tkocz, M. Kus and K. Zyczkowski, Extremal spacings between eigenphases of random unitary matrices and their tensor products, Phys. Rev. E 88, 052902 (2013).
\bibitem{Sel}A. Selberg, Bemerkninger om et multipelt integral, Norsk. Mat. Tidsskr. 24 (1944), 71-78.
\bibitem{So}A. Soshnikov, Statistics of extreme spacing in determinantal random point processes. Moscow Math. J., vol.5, No.3, 705-719, (2005).
\bibitem{So2}A. Soshnikov,  Determinantal random point fields,  Russian Math. Surveys 55, no 5, 923-975 (2000).
\bibitem{SJ}D. Shi and Y. Jiang, Smallest gaps between eigenvalues of random matrices with complex Ginibre, Wishart and universal unitary ensembles, arXiv:1207.4240.
\bibitem{V}J. Vinson,  Closest spacing of eigenvalues. Ph.D. thesis, Princeton University, 2001.

\end{thebibliography}
\end{document}